\documentclass[12pt]{article}

\usepackage{arxiv}
\usepackage{amsmath,amsthm,amssymb}
\usepackage{epsfig,multirow,makecell,tikz-cd,comment,xcolor,graphics}
\usepackage{changepage}
\usepackage{placeins}
\usepackage{hyperref}
\hypersetup{
  colorlinks=true,
  linkcolor=blue!70!black,
  citecolor=green!40!black,
  urlcolor=blue!70!black
}
\usepackage{thmtools}
\usepackage{amsfonts}       
\usepackage[verbose=true]{microtype}
\usepackage{graphicx}
\usepackage{tikz}
\usepackage{mathtools}
\usepackage{booktabs}
\usetikzlibrary{matrix,positioning}
\usepackage[caption=false]{subfig} 
\usepackage{caption}
\usepackage{thm-restate}

\theoremstyle{plain}
\newtheorem{theorem}{Theorem}[section]
\newtheorem{lemma}[theorem]{Lemma}
\newtheorem{proposition}[theorem]{Proposition}
\newtheorem{corollary}[theorem]{Corollary}
\newtheorem{statement}[theorem]{Statement}

\theoremstyle{definition}
\newtheorem{definition}[theorem]{Definition}

\theoremstyle{remark}
\newtheorem*{remark}{Remark}

\usepackage[capitalize,nameinlink]{cleveref}
\crefname{statement}{Statement}{Statements}

\usepackage[shortlabels]{enumitem}
\usepackage{rotating}
\usepackage{fix-cm}
\usepackage{relsize}
\tikzset{fontscale/.style = {font=\relsize{#1}}
    }

\usetikzlibrary{decorations.markings}
\usetikzlibrary{calc}

 \usepackage{tikz}
 \usetikzlibrary{arrows}

\usepackage[verbose=true,letterpaper]{geometry}
\AtBeginDocument{
  \newgeometry{
    textheight=8.25in,
    textwidth=6.151in,
    top=1.3in,
    headheight=14pt,
    headsep=25pt,
    footskip=30pt
  }
}
\usepackage{fancyhdr}
\fancyheadoffset{0pt} 
\title{Persistent Homology with Path-Representable Distances on Graph Data\thanks{Preprint. Under review.}}

\author{
 \textsuperscript{ab}Eunwoo Heo\thanks{these authors contributed equally to this work}
 \and 
 \textsuperscript{ab}Byeongchan Choi\footnotemark[2]
 \and
 \textsuperscript{ab}Jae-Hun Jung\thanks{Corresponding author \\ \\ Email: hew0920@postech.ac.kr(Eunwoo Heo), bcchoi@postech.ac.kr(Byeongchan Choi), 
 jung153@postech.ac.kr(Jae-Hun Jung)}}

\affiliations{\textsuperscript{a} Department of Mathematics, Pohang University of Science \& Technology, Pohang 37673, Korea; \\
\textsuperscript{b} POSTECH Mathematical Institute for Data Science (MINDS), Pohang University of Science \& Technology, Pohang 37673, Korea; \\
}

\usepackage{algorithm}
\usepackage{algorithmic}
\Crefname{ALC@unique}{Line}{Lines}

\begin{document}

\maketitle

\begin{abstract}
Persistent homology (PH) has been widely applied to graph data to extract topological features.
However, little attention has been paid to how different distance functions on a graph affect the resulting persistence barcodes and their interpretations.
In this paper, we define a class of distances on graphs, called \textit{path-representable distances}, and investigate structural relationships between their induced persistent homologies.
In particular, we identify a nontrivial injection between the $1$-dimensional barcodes induced by two commonly used graph distances: the unweighted and weighted shortest-path distances.
We formally establish sufficient conditions under which such embeddings arise, focusing on a subclass we call \textit{cost-dominated} distances.
The injection property is shown to hold in $0$- and $1$-dimensions, while we provide counterexamples for higher-dimensional cases.
To make these relationships measurable, we introduce the \textit{total persistence difference} (TPD), a new topological measure that quantifies changes between filtrations induced by cost-dominated distances on a fixed graph.
We prove a stability result for TPD when the distance functions admit a partial order and apply the method to the SNAP EU Research Institution E-Mail dataset. 
TPD captures both periodic patterns and global trends in the data, and shows stronger alignment with classical graph statistics compared to an existing PH-based measure applied to the same dataset. 
\end{abstract}


\Classifications{\small AMS 00A65, 55N31.}

\section{Introduction}\label{sec:Introduction}

Topological data analysis (TDA), developed recently, has proven its usefulness in various applications~\cite{Cohen, Carlsson,ZC2005,EHarer}.
Persistent homology (PH) is one of the primary tools in TDA, using homological characteristics of the data at various scales~\cite{Carlsson}. The PH method involves constructing the complex in different scales sequentially, capturing how homology changes as the complex is being built and refined. This procedure is known as {\it filtration}. These hierarchical changes of homology enable the PH method to provide the topological inference of the data.
 
For the PH method, it is crucial to transform the given data into a suitable embedding space. 
Graph representation is widely used in network science for PH analysis~\cite{Aktas}, including the brain network~\cite{sizemore_netneurtda, stolz_phfunc, leebraintda}, financial science~\cite{gideafinance, GIDEA2018820, MAJUMDAR2020113868}, text mining~\cite{gholizadehtext, gholizadehtext2} and music~\cite{mehmetmakam, tran2023topological}.

Once the appropriate transform is identified, applying the PH method requires defining a proper metric for the transformed geometric object, determining the distance between points in the point cloud or between nodes in the graph representation.
There is no single way to define this metric or distance as there can be multiple definitions.
Definitions of such distances have been empirically established in different fields of study, yet there is a lack of research on how these varying definitions of distance impact PH calculations.
For a brief functorial overview of common graph-to-persistence pipelines in TDA, see Appendix~\ref{app:tda-graphs}.

In this paper, we consider a connected weighted graph \(G = (V, E, W_E)\), where \(W_E: E \rightarrow \mathbb{R}^{> 0}\) is a function that assigns weights to the edges. A natural definition of the distance \(d_{\text{weight}}(v, w)\) between any two points \(v, w \in V\) on the graph can be chosen as the minimum sum of weights along the paths connecting \(v\) and \(w\). This definition of distance, namely the shortest path distance, satisfies the metric conditions and has been thoroughly studied both in theoretical terms~\cite{ahujashortpath, cherkasskypriority, FREDMAN1994533} and in computational terms~\cite{Cherkassky1996, goldbergsimple, zhushortest}, as well as in approximations~\cite{RIZI2018IEEE, POTAMIAS2009ACM}. The definition is also widely used in broad fields of computer science and is especially powerful in improving the performance of neural networks~\cite{NEURIPS2020_2f73168b, pmlr-v198-abboud22a, JI2021IEEE, pmlr-v202-michel23a, li2021distance, abboud2022shortest}.

On the other hand, the shortest path as the path connecting two vertices with the least number of edges can also be considered. 
In a similar sense, there are data in which the relevancy between two nodes is more explained in the path passing through the minimum number of edges than in that with minimum sum of weights.
In such a case, it is natural to define the distance between two vertices using such paths. The first approach was conducted in previous research~\cite{tran2023topological}. 
We canonically reformulated the definition of distance \(d_{\text{edge}}(v, w)\) between two points on the graph to be the minimum sum of weights between the paths connecting \(v\) and \(w\) with the minimum number of edges.

In this research we found an interesting relation between \(d_{\text{weight}}\) and \( d_{\text{edge}}\). That is, for \(d_{\text{weight}} \leq d_{\text{edge}}\), we have \( \textbf{bcd}_1(d_{\text{weight}}) \subseteq \textbf{bcd}_1(d_{\text{edge}})\), where \(\textbf{bcd}_1\) is the $1$-dimensional barcode. 
This property is quite unusual and is generally not expected to hold, as depicted in~\cref{fig:counter_example_embedding}. 
To explore this, we introduce a more general concept of distance, termed \textit{path-representable distance}, and identify a class of \textit{cost-dominated} distances for which this inclusion always holds.
Furthermore, in this paper, we prove that such an injection relation holds among any path-representable cost-dominated distances for $1$-dimensional homology.

Beyond the existence of this injective structure, it is natural to ask how much the $1$-dimensional persistence changes as the distance is altered.
Given two cost-dominated path-representable distances \(d_i \le d_j\) on the same graph, we define the \textit{total persistence difference} (TPD) as a nonnegative quantity that compares the corresponding $1$-dimensional barcodes cycle by cycle.
Each birth edge that appears for both distances contributes the increase in its death value under \(d_j\) relative to \(d_i\), and each birth edge that appears only for \(d_j\) contributes the persistence of the associated interval.
In this way TPD measures how the $1$-dimensional structure extends when the distance on the graph is changed.

We prove that TPD enjoys a Lipschitz stability with respect to the pointwise difference between the underlying distances.
More precisely, for cost-dominated path-representable distances \(d_i \le d_j\) we obtain an upper bound of the form
\[
  \mathrm{Diff}_p\bigl(\textbf{bcd}_1(G,d_j),\,\textbf{bcd}_1(G,d_i)\bigr)
  \le \gamma \,\|d_j-d_i\|_\infty,
\]
where the constant \(\gamma\) depends only on the number of vertices and edges of the graph.

Finally, we investigate how TPD behaves on real temporal graph data.
We apply our framework to the SNAP EU Research Institution E-mail dataset, where for each day we build a weighted graph from all e-mails exchanged during that day and consider two natural path-representable distances on this graph.
On this dataset TPD is typically close to zero on structurally simple days, and becomes larger when the daily graph has higher degree and clustering.
We also compare TPD with a previously proposed PH-based measure computed from overlapping time windows by reproducing the pipeline of~\cite{hajij2018visual}.
In this comparison, TPD shows stronger Pearson correlations with basic graph statistics and better reflects periods of low activity, while the Wasserstein-based day-to-day distance does not decay on weekends where the graph structure is sparse or trivial.

This paper is composed of the following sections.
In~\cref{sec:Distances defined across different domains and their persistent homology}, we explain distinct distance definitions over the given graph, particularly the weight- and edge-minimal distances and the possible relations among them.
In~\cref{sec:Main theory}, we introduce the concept of path-representable distance, a generalization of \(d_{\text{weight}}\) and \(d_{\text{edge}}\) and we introduce the main theory related to path-representable distances.
In~\cref{sec:Proofs of main theory}, we rigorously prove the main propositions and theorem.
In~\cref{sec:Analysis of dimensional variations of the main theorem}, we make an experiment to explain how the main theorem works in higher dimensions.
In~\cref{sec:Data Analysis}, we define the total persistence difference and establish its stability theorem.
In~\cref{sec:Applications}, we apply our method to the e-mail dataset and compare TPD with existing PH-based quantities.
Finally, in~\cref{sec:Conclusion}, we provide a brief concluding remark.
\section{Different distance definitions and corresponding persistent homology}\label{sec:Distances defined across different domains and their persistent homology}

\subsection{Different distance definitions}\label{subsec:Distances defined across different domains}

Let $G=(V,E,W_E)$ be a connected weighted graph.

\begin{definition}\label{def: distance_weight}
    Let \( G=(V,E,W_E) \) be a connected weighted graph. Define the distance \( d_{\text{weight}}: V \times V \rightarrow \mathbb{R}^{\ge 0} \) such that for any two distinct vertices \( v \) and \( w \) in \( V \),
    \[
    d_{\text{weight}}(v,w) = \min_p \left\{ \sum_{e \in E(p)} W_E(e) \mathrel{\bigg|} p \text{ is a path from } v \text{ to } w \right\},
    \]
    and $d_{\text{weight}}(v,w) = 0$ if $v=w$, where \(E(p)\) denotes the set of edges composing the path \(p\).
\end{definition}

Next, consider the definition associated with the paths involved in terms of the minimum number of edges. 
\begin{definition}\label{def: distance_edge}
    Let \( G=(V,E,W_E) \) be a connected weighted graph. 
    Define the distance \( d_{\text{edge}}: V \times V \rightarrow \mathbb{R}^{\ge 0} \) such that for any two distinct vertices \( v \) and \( w \) in \( V \),
    \[
    d_{\text{edge}}(v,w) = \min_p \left\{ \sum_{e \in E(p)} W_E(e) \mathrel{\bigg|} p \text{ is a path between } v \text{ and } w \text{ such that } |E(p)| = m_{vw} \right\},
    \]
    and $d_{\text{edge}}(v,w) = 0$ if $v=w$, where \(E(p)\) denotes the set of edges composing the path \(p\), and \( m_{vw} \) represents the minimum number of edges among all paths from \( v \) to \( w \).
\end{definition}

\begin{proposition}\label{prop:d_weight<=d_edge}
    Let $d_{\text{edge}}$ and $d_{\text{weight}}$ be the distances defined above on a connected weighted graph \( G=(V,E,W_E) \).
    Then the following inequality holds  $$ d_{\text{weight}}(v,w) \leq d_{\text{edge}}(v,w) \text{ for any } v,w \in V$$
\end{proposition}

\begin{proof}
    Take any pair of two distinct vertices $v$ and $w$ in $V$. 
    Let $\mathcal{P}(v,w)$ be the set of all paths between $v$ and $w$. Define $\mathcal{P}_{min}(v,w)$ as the subset of $\mathcal{P}(v,w)$ containing all paths with the minimum number of edges, i.e.,
    \(
    \mathcal{P}_{min}(v,w) = \left\{ p \in \mathcal{P}(v,w) \mid |E(p)| = m_{vw} \right\},
    \)
    where $m_{vw}$ represents again the minimum number of edges among all paths in $\mathcal{P}(v,w)$.
    Then, by definition, we have 
    \[
    d_{\text{weight}}(v,w) 
    = \min_{p\in \mathcal{P}(v,w)} \left\{ \sum_{e \in E(p)} W_E(e) \right\} \
    \le 
    \min_{p\in \mathcal{P}_{min}(v,w)} \left\{ \sum_{e \in E(p)} W_E(e) \right\} = d_{\text{edge}}(v,w).
    \]
\end{proof}

\subsection{Persistent homology}\label{subsec:Persistent homology}

Consider a connected weighted graph $G=(V,E,W_E)$ with its distance $d:V \times V \rightarrow \mathbb{R}^{\ge 0}$.
Define an abstract simplicial complex $\mathbb{X}$ as the power set of the vertex set $V$, denoted by $\mathbb{X} \coloneqq \mathcal{P}(V)$.
Consider a real-valued function $f: \mathbb{X} \rightarrow \mathbb{R}$ defined as
\[
f(\sigma) = \max \{ d(v,w) \mid (v,w) \subseteq \sigma \text{ for } v,w \in V\}
\]
for any $k$-simplex $\sigma \in \mathbb{X}$ whenever $k \ge 1$, otherwise $f(\sigma) = 0$.
For a sequence of real numbers $0 = a_0 \le a_1 \le \ldots \le a_m$, if we denote $\mathbb{X}_{a_l} \coloneqq f^{-1}((-\infty,a_l])$, then we have a filtration 
\[
 V = \mathbb{X}_{a_0} \hookrightarrow \mathbb{X}_{a_1} \hookrightarrow \cdots \hookrightarrow \mathbb{X}_{a_m}= \mathbb{X}.
\]
Note that $\mathbb{X}_{a_0} = V$ and $\mathbb{X}_{a_m}= \mathcal{P}(V)$.
This filtration is called the Vietoris--Rips (Rips) filtration.
The filtration induced by $f$ allows us to study topological changes in the simplicial complex $\mathbb{X}$ as the parameter $a_l$ varies. 
As the threshold $a_l$ increases, more simplices are included in the filtration, revealing more complex interactions among the vertices based on their distances. 
This process helps to understand the shape of the data represented by the graph $G$.

Persistent homology captures and quantifies these topological changes throughout different scales. For each $k$-dimensional simplex, we track when it appears and becomes part of a larger homological feature as $a_l$ increases. 
This is formally achieved by computing the homology groups $H_k(\mathbb{X}_{a_l})$ at each stage of the filtration, using homological algebra over the coefficient field $\mathbb{F}$. 
Persistent homology tracks the birth and death of homological features such as connected components ($k=0$), loops ($k=1$), and voids ($k=2$) as they appear and disappear through the filtration.

A persistence barcode $\textbf{bcd}_k(d)$, for each homological dimension $k$, is a visual representation of homological features as a multiset. 
They can be represented as follows:
\[
\textbf{bcd}_k(d) = \{[\beta, \delta] \mid \beta, \delta \in \mathbb{R}, \beta < \delta\}.
\]
Each barcode consists of intervals, where the beginning of an interval corresponds to a \textit{birth} (created when a feature appears in the filtration) and the end corresponds to a \textit{death} (created when a feature is no longer visible in the filtration as a distinct entity). 
These intervals provide a powerful summary of the topological features of the data across multiple scales, encapsulating both the geometry and topology of the underlying graph structure.
By analyzing these barcodes, we can identify significant topological features that persist over a wide range of scales, which correspond to important structural properties of the data. 
This methodology of tracking and summarizing the change of topological features via persistent homology provides insights into the complex structures hidden in high-dimensional data.
In this paper from now on, we fix the coefficient field $\mathbb{F}$ to be $\mathbb{Z}/2\mathbb{Z}$.

In~\cite{cohen2006vines}, an algorithm for computing persistence barcodes was presented.
Let $\mathbf{D}$ be the boundary matrix.
Each row and column of the matrix $\mathbf{D}$ corresponds to a simplex,
and the matrix $(\mathbf{D}_{i,j})$ is defined as 
$\mathbf{D}_{i,j} =\begin{cases}
 1 & \text{ if } \sigma_i \in \partial\sigma_j \\
 0 & \text{ otherwise}
\end{cases}$.
The algorithm performs the column operations known as reduction to obtain the reduced binary matrix $\mathbf{R}$, starting from $\mathbf{D}$.
Let $\text{low}_{\mathbf{R}}(j)$ represent the row index of the last nonzero element in the $j$th column of $\mathbf{R}$, 
or be undefined if the $j$th column consists solely of zeros. 
The column operation algorithm proceeds as follows: 

\begin{algorithm}
\caption{$\mathbf{R}=\mathbf{D}\mathbf{V}$ decomposition}
\label{alg:R=DV}
\begin{algorithmic}[1]
\STATE{Define $\mathbf{R} \coloneqq \mathbf{D} \coloneqq \text{the boundary matrix}$}
\FOR{$j = 1$ to $n$}
    \WHILE{there exists $j' < j$ such that $\text{low}_{\mathbf{R}}(j') = \text{low}_{\mathbf{R}}(j)$}
        \STATE{Update $\mathbf{R}[:, j] \coloneqq \mathbf{R}[:, j] - \mathbf{R}[:, j']$}
    \ENDWHILE
\ENDFOR
\RETURN $\mathbf{R}$
\end{algorithmic}
\end{algorithm}
The resulting matrix $\mathbf{R}$ is called the reduced matrix. 
Since the algorithm only uses the column operation from left to right, it is equivalent to the decomposition $\mathbf{R}=\mathbf{D}\mathbf{V}$ with an upper triangular matrix $\mathbf{V}$.
The important things to know here are that 
\begin{enumerate}
  \item When the $i$ th column of $\mathbf{R}$ corresponding to the $k$-simplex $\sigma_i$ is zero, there is a cycle $\tau$  containing $\sigma_i$ in $\mathbb{X}_b $, where $b$ is the filtration value of $\sigma_i$.
  \item If there is a $(k+1)$-simplex $\sigma_j$ such that $i = \text{low}_{\mathbf{R}}(j)$, then $\sigma_j$ causes the cycle $\tau$ to die, where the death value is the value of $\sigma_j$.
  If there is no such $j$, the cycle $\tau$ will persist infinitely.
  \item Suppose that the $i$ th column of $\mathbf{R}$ corresponding to the $k$-simplex $\sigma_i$ is zero.
  Let $f$ be the real-valued function induced by the distance $d$ as described in~\cref{subsec:Persistent homology}.
  If there is a $(k+1)$-simplex $\sigma_j$ such that $i = \text{low}_{\mathbf{R}}(j)$ with $f(\sigma_i) < f(\sigma_j)$,
  then $[f(\sigma_i), f(\sigma_j)]$ is in $\textbf{bcd}_k(d)$.
  Also, if there is no $\sigma_j$ such that $i = \text{low}_{\mathbf{R}}(j)$,
  then $[f(\sigma_i), \infty ]$ is an element of $\textbf{bcd}_k(d)$.
\end{enumerate}

Therefore, the $k$-dimensional persistence barcode $\textbf{bcd}_k(d)$ is a multi-set 
\[\{ [f(\sigma_i), f(\sigma_j)]  \mid \sigma_i \text{ is a $k$-simplex with }i = \text{low}_{\mathbf{R}}(j), f(\sigma_i) < f(\sigma_j) \}\]
\[
\bigcup \{ [f(\sigma_i), \infty]  \mid \sigma_i \text{ is a $k$-simplex, there is no $\sigma_j$ such that } i = \text{low}_{\mathbf{R}}(j)\}.
\]

When computing the persistence barcode of $\mathbb{X}$, the complexes $\mathbb{X}_{a_i}$ are constructed based on the filtration value $a_i$. Therefore, simplices with the same value $a_i$ appear simultaneously in $\mathbb{X}_{a_i}$, yet their order in the column of the boundary matrix $\mathbf{D}$ should differ. However, this order does not affect the computation of the persistence barcode. This property is known as \textit{permutation invariance}.

Consider the ordering function $\pi$ on $\mathbb{X}$ that determines the order of simplices with the same value for each $\mathbb{X}_{a_i}$. 
We call this function the \textit{simplex ordering on} $\mathbb{X}$ from the function $f$. 
Let $\mathbf{D}_{\pi}$ be the boundary matrix with $\pi$. Then $\mathbf{D}_{\pi}$ is uniquely determined if we fix a simplex ordering $\pi$ on $\mathbb{X}$.
Also, if $\mathbf{D}_{\pi}$ is uniquely determined, then 
every $[f(\sigma_i), f(\sigma_j)]$ in $\textbf{bcd}_k(d)$ is uniquely determined by $\sigma_i$.
We obtain the set version of the persistence barcode as
\[
\textbf{bcd}^{\pi}_k(d) = \{[f(\sigma_i), f(\sigma_j)]_{\sigma_i} \mid [f(\sigma_i), f(\sigma_j)] \in \textbf{bcd}_k(d) \} ,
\]
where $\textbf{bcd}^{\pi}_k(d)$ denotes the barcode with the ordering function $\pi$.
For each element $[f(\sigma_i), f(\sigma_j)]_{\sigma_i}$ of $\textbf{bcd}^{\pi}_k(d)$, $\sigma_i$ is called the \textit{birth $k$-simplex} of $[f(\sigma_i), f(\sigma_j)]_{\sigma_i}$.

\subsection{Interrelations between two persistence barcodes}\label{subsec:Interrelations between two persistence barcodes}

Given the connected weighted graph $G=(V, E, W_E)$, the distances $d_{\text{weight}}$ and $d_{\text{edge}}$ introduced in~\cref{subsec:Distances defined across different domains} are generally distinct. 
Although \(d_{\text{weight}}\) is a metric and \(d_{\text{edge}}\) is a semimetric, we found that there is an interesting relation between the $1$-dimensional persistence barcodes by these distances as described in~\cref{Thm: main theorem}.

\begin{theorem}\label{Thm: main theorem}
Let $G=(V,E,W_E)$ be a connected weighted graph. 
Then there exists an injective function $\varphi : \textbf{bcd}_1(d_{\text{weight}}) \rightarrow \textbf{bcd}_1(d_{\text{edge}})$ defined as $\varphi([\beta,\delta])=[\beta, \delta ']$ such that $\delta \le \delta '$ for any $[\beta,\delta] \in \textbf{bcd}_1(d_{\text{weight}})$.   
\end{theorem}

\begin{figure}[!ht]
    \centering
    \includegraphics[width=0.9\textwidth]{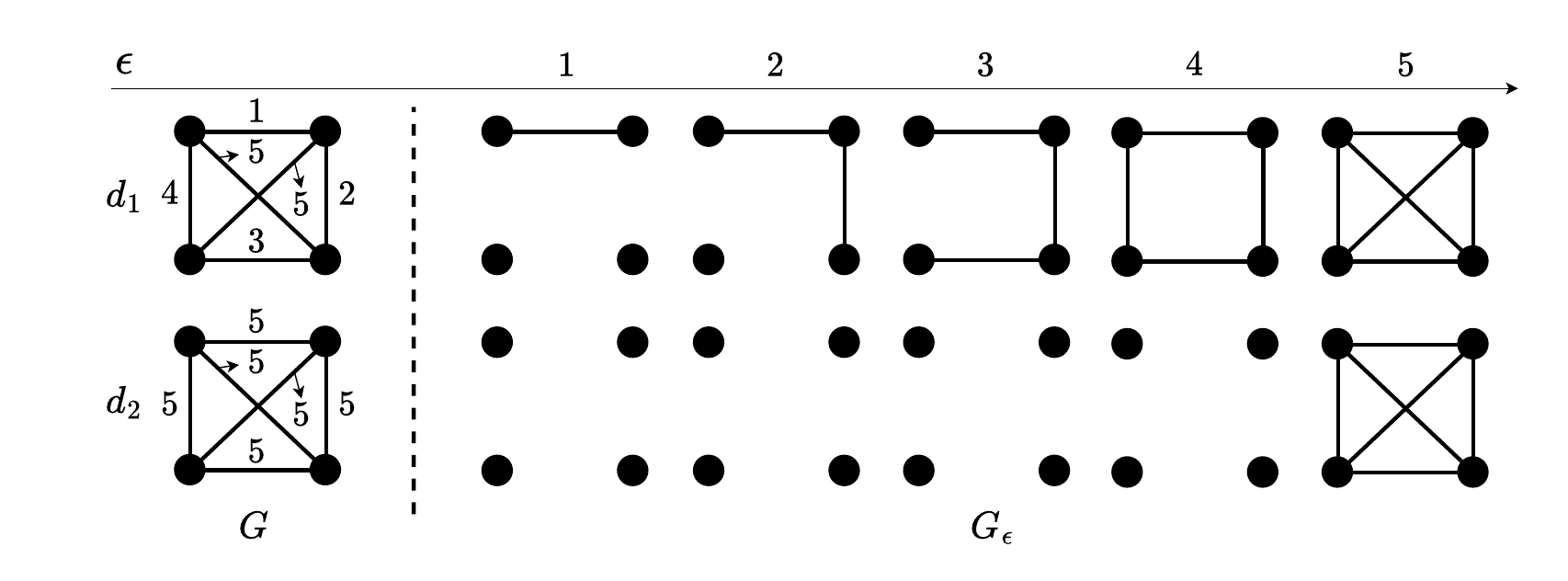} 
    \caption{An example where \(d_1 \leq d_2\) yet \(\textbf{bcd}_1(d_1) \not\subseteq \textbf{bcd}_1(d_2)\), with \(\textbf{bcd}_1(d_1) =\{ [4,5] \}\) and \(\textbf{bcd}_1(d_2) =\emptyset \).
    The subgraph \( G_{\epsilon} \) of \( G \) is defined as \( G_{\epsilon} = (V, E_{\epsilon}) \), where \( E_{\epsilon} = \{ e \in E \mid W_E(e) \leq \epsilon \} \).
    }
    \label{fig:counter_example_embedding}
\end{figure}

As illustrated in~\cref{fig:counter_example_embedding}, the presence of an inequality between the two arbitrary distances does not guarantee an injection between their $1$-dimensional persistence barcodes in general. However, \cref{Thm: main theorem} shows the existence of a pair of distance definitions that could provide an injective relationship in the $1$-dimensional barcode. For example, a possible pair would be $d_{\text{edge}}$ and $d_{\text{weight}}.$

\cref{Thm: main theorem} implies that \(\textbf{bcd}_1(d_{\text{edge}})\) always contains more elements of persistence barcodes than \(\textbf{bcd}_1(d_{\text{weight}})\). 
From this perspective, examining \(d_{\text{edge}}\) provides access to information that cannot be obtained through \(d_{\text{weight}}\) alone. 
Furthermore, analyzing the significance of the remaining elements of \(\textbf{bcd}_1(d_{\text{edge}})\) that are not matched with the injective function \(\varphi\) in~\cref{Thm: main theorem} needs further investigation.

\begin{figure}[!ht]
    \centering
    \includegraphics[width=0.9\textwidth]{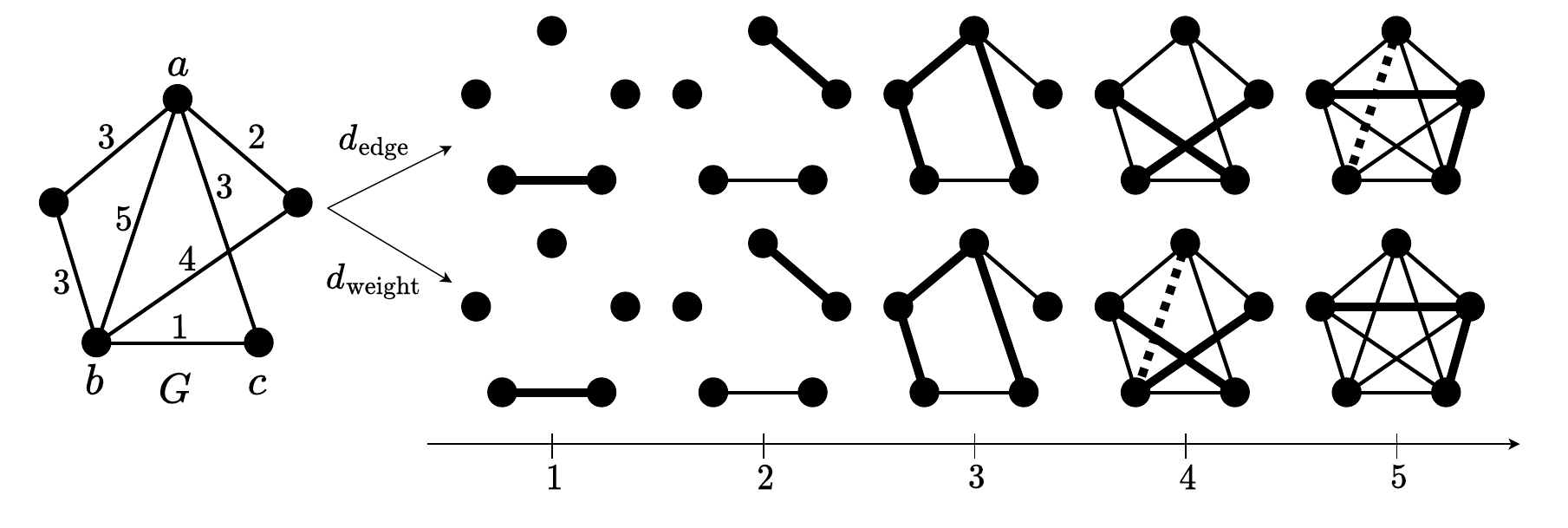} 
    \caption{
    An example where $\textbf{bcd}_1(d_{\text{edge}})$ and $\textbf{bcd}_1(d_{\text{weight}})$ differ yet they form a hierarchical structure within $1$-dimensional hole structures. Bold lines represent newly formed edges, while dashed lines indicate edges that appear at different values.
    }
    \label{fig:Toy_main_example}
\end{figure}

\cref{fig:Toy_main_example} shows the order in which edges appear in a Rips filtration, based on $d_{\text{edge}}$ and $d_{\text{weight}}$. When calculating $d_{\text{edge}}$, the direct path from \(a\) to \(b\) yields a value of $5$, but for $d_{\text{weight}}$, the path from \(a\) to \(b\) via \(c\) results in a value of $4$.
Due to this difference, $\textbf{bcd}_1(d_{\text{edge}}) = \{ [3,4] , [4,5] \}$ and $\textbf{bcd}_1(d_{\text{weight}}) = \{ [3,4] \}$ differ and the injection $\varphi : \textbf{bcd}_1(d_{\text{weight}}) \rightarrow \textbf{bcd}_1(d_{\text{edge}})$ from~\cref{Thm: main theorem} satisfies $\varphi([3,4]) = [3,4]$.

\section{Main theory}\label{sec:Main theory}

In this section, we prove~\cref{Thm: main theorem} by defining more general distance notions and also prove the theorem that is more general than~\cref{Thm: main theorem}. In particular, the distances \(d_{\text{weight}}\) and \(d_{\text{edge}}\) introduced in~\cref{subsec:Distances defined across different domains} are special cases of a more general distance definition, termed as the \textit{path-representable distance}, which we will define and explore in this section.

\subsection{Path-representable distance}\label{sec:Path-representable distance}

Let $G=(V, E, W_E)$ be a connected weighted graph. 
For any two vertices $v$, $w$ in $V$, $P(v,w)$ is the set of paths from $v$ to $w$ that have no repeated vertex, and $\mathcal{P}$ is the union of $P(v,w)$ along all $v,w \in V$. 
That is, 
\[
\mathcal{P} = \bigcup_{v,w \in V} P(v,w) 
\]
A \textit{path choice function} is a function $g: V \times V \rightarrow \mathcal{P}$ satisfying $g(a,b) \subseteq g(v,w) \in P(v,w)$ for any vertices $v$ and $w$ in $G$ and any vertices $a$ and $b$ in $g(v,w)$. This property is called \textit{consistency}.

\begin{definition}
    Consider a connected weighted graph $G=(V, E, W_E)$. 
    A function $d$ on $V \times V$ is called a \textit{path-representable distance} if there is a path choice function $g$ such that 
    \[
    d(v,w) = \sum_{e \in g(v,w)} W_E(e) \text{ if } v \neq w
    \]
    and $d(v,w) = 0$ if $v=w$.
\end{definition}

In fact, the consistency property is natural for the definition of the path choice function.
The consistency property is equivalent to stating that the chosen path must have \textit{an optimal substructure}.
In computer science, a problem is said to have an optimal substructure if its solution can be constructed by combining solutions to smaller subproblems.
For example, in dynamic programming, it is crucial to determine whether or not a given problem possesses an optimal substructure. This characteristic allows a large problem to be divided into smaller sub-problems, each with an optimal solution, which can then be combined to construct the optimal solution for the entire problem, thereby enabling efficient problem resolution.

Consider a connected weighted graph $G=(V,E,W_E)$ and its path choice function $g$ on $G$.
For an edge $e=(v,w) \in E$, we say that $g$ is \textit{locally weight-dominated} by $e$ if $W_E(e) >  W_E(e')$ for any $e' \in g(v,w)$. 
In addition, we say that $g$ is \textit{locally cost-dominated} by $e$ if $W_E(e) \ge \sum_{e' \in g(v,w)} W_E(e')$.
Note that if $g$ is locally cost-dominated by $e$, then $g$ is locally weight-dominated by $e$.

We define that $g$ is \textit{weight-dominated} if $g$ is locally weight-dominated by any $e=(v,w) \in E$ whenever $g(v,w) \neq e$ and $g$ is \textit{cost-dominated} if $g$ is locally cost-dominated by any $e=(v,w) \in E$ whenever $g(v,w) \neq e$.
These conditions mean that there must be a reason to choose a path $g(v,w)$ instead of the single-edge path $e$ in $G$.

\begin{definition}[Dominated distance]
    Consider a connected weighted graph $G=(V, E, W_E)$ and its path-representable distance $d$ represented by the path choice function $g$ on $G$.
    We refer to $d$ as \textit{weight-dominated} if $g$ is weight-dominated and as \textit{cost-dominated} if $g$ is cost-dominated.
\end{definition}

\begin{proposition}\label{dweight_dedge_elements}
    The two distances $d_{\text{weight}}$ and $d_{\text{edge}}$ defined in~\cref{def: distance_weight,def: distance_edge} are cost-dominated distances.
\end{proposition}

\begin{proof}
(i) First, we show that \( d_{\text{weight}} \) and \( d_{\text{edge}} \) are path-representable distances.
Consider \( d_{\text{weight}} \).
Let \( V = \{ v_1, v_2, \ldots, v_n \} \). 
Denote \( p_{1,k} \) by the minimal path from \( v_1 \) to \( v_k \) such that \(d_{\text{weight}}(v_1, v_k) = \sum_{e \in E(p_{1,k})} W_E(e)\)  
for \( k = 2,3, \ldots , n \).
Consider \( p_{1,2} \) and \( p_{1,3} \).
Suppose \( p_{1,2} \) and \( p_{1,3} \) meet at the midpoint \( z \neq v_1 \).
Then, denote the subpath from \( v_1 \) along \( p_{1,2} \) to \( z \) by \( p_{1,2}|_{z} \), and the subpath from \( v_1 \) along \( p_{1,3} \) to \( z \) by \( p_{1,3}|_{z} \).
By the minimality of \( p_{1,2} \) and \( p_{1,3} \), 
\(\sum_{e \in E(p_{1,2}|_{z} )} W_E(e) = \sum_{e \in E(p_{1,3}|_{z} )} W_E(e)\).
This is because if one side were smaller, we could find a shorter path by replacing the other.
Now, redefine \( p_{1,3} \) as a new path starting at \( v_1 \), following \( p_{1,2}|_{z} \) to \( z \), and then from \( z \) to \( v_3 \) along the path of \( p_{1,3} \).
Repeat this process until the midpoint where \( p_{1,2} \) and \( p_{1,3} \) meet disappears.

This process ensures one of the following:
(a) \( p_{1,2} \) and \( p_{1,3} \) meet only at \( v_1 \) or not at all, or
(b) they start from \( v_1 \), overlap only on a common subpath, and do not meet anywhere else until the end. We call the above replacement process a \((*)\)-replacement.

Next, perform the \((*)\)-replacement for \( p_{1,4} \) with respect to \( \{ p_{1,2}, p_{1,3} \} \),
and repeat this until \( p_{1,n} \) is \((*)\)-replaced with respect to \( \{ p_{1,2}, p_{1,3}, \ldots, p_{1,n-1} \} \).
Then, take \( p_{2,1} \), reverse its direction to be the same as \( p_{1,2} \),
and \((*)\)-replace \( p_{2,3} \) with respect to \( p_{2,1} \).
Continue to \( p_{2,4} \) with respect to \( \{ p_{2,1}, p_{2,3} \} \),
and repeat this process until \( p_{2,n} \) is \((*)\)-replaced with respect to \( \{ p_{2,1}, p_{2,3}, \ldots, p_{2,n-1} \} \).
By continuing this process, we ensure that for any distinct \( i\) and \(j\), 
the consistency property for \( p_{i,j} \) is satisfied.
Finally, define the path choice function \( g \) as \( g(v_i, v_j) = p_{i,j} \), 
thus proving that \(d_{\text{weight}}(v_i, v_j) = \sum_{e \in g(v_i, v_j)} W_E(e)\) .

The \((*)\)-replacement process can also be applied in the case of \(d_{\text{edge}}\). 
If there is a midpoint \(z\), both the number of edges in the two different paths with the same starting points and the sum of the weights up to the midpoint must be the same. Thus the replacement can be done similarly. 
Therefore, \(d_{\text{edge}}\) becomes the path-representable distance by undergoing the similar procedure.

(ii) Second, we demonstrate that the path choice function \(g\) of \(d_{\text{weight}}\) in (i) is cost-dominated. 
Consider any edge \(e = (v, w)\) with \(g(v, w) \neq e\).
Since the edge \(e\) itself also constitutes a path from \(v\) to \(w\), by considering the minimality of \(g(v, w)\), we have \(W_E(e) \geq \sum_{e' \in g(v, w)} W_E(e')\).
Additionally, \(d_{\text{edge}}\) is cost-dominated by its definition.
\end{proof}

\subsection{Toy example of path-representable distance}

\begin{figure}[!ht]
    \centering
    \includegraphics[width=0.8\textwidth]{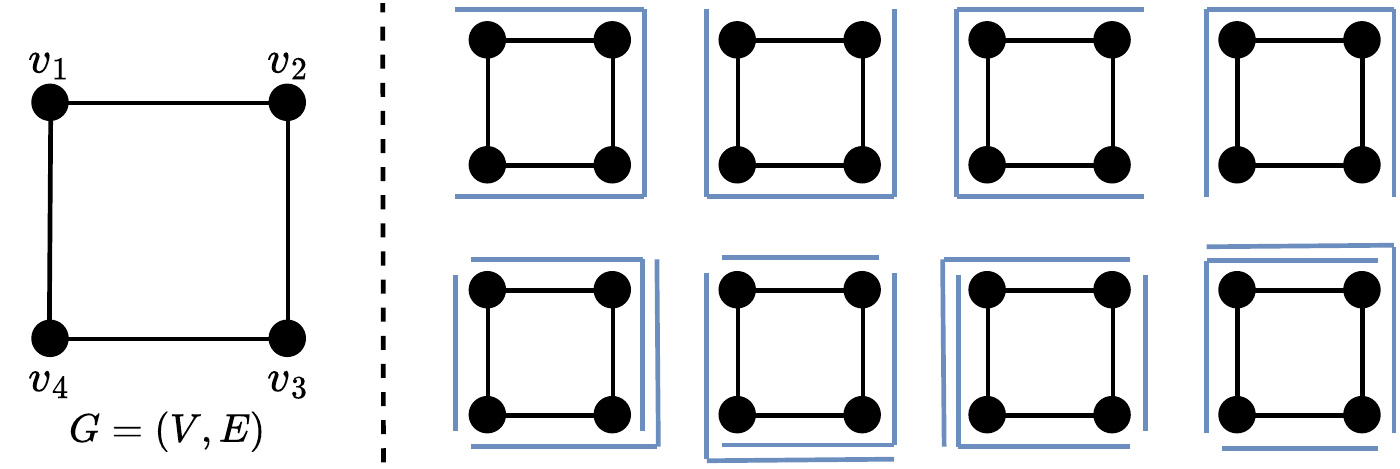} 
    \caption{The graph 
    $G=(V,E)$ and all of its possible path choice functions represented by their maximal paths.}
    \label{fig:Example_path_repre}
\end{figure}

Consider a connected weighted graph $G=(V,E, W_E)$ with \(V = \{v_1, v_2, v_3, v_4\}\) and \(E = \) \(\{ (v_1,v_2),\) \( (v_2,v_3) \), \( (v_3,v_4) \),\((v_1,v_4) \}\) as shown in the left figure of~\cref{fig:Example_path_repre}.
How many path-representable distances can we make for any path choice function $g$ on $G$? 
A maximal path can represent any path choice function $g$ on $G$, where the maximal path refers to a path that is maximal with respect to the inclusion.
The right figure of~\cref{fig:Example_path_repre} represents the maximal paths indicated by the given path choice function on graph \(G\).
The maximal paths are well-defined due to the consistency property of the path choice function.
In this case, a total of eight possible path choice functions can be defined on \(G\), as shown in~\cref{fig:Example_path_repre}.

\begin{figure}[!ht]
    \centering
    \includegraphics[width=0.9\textwidth]{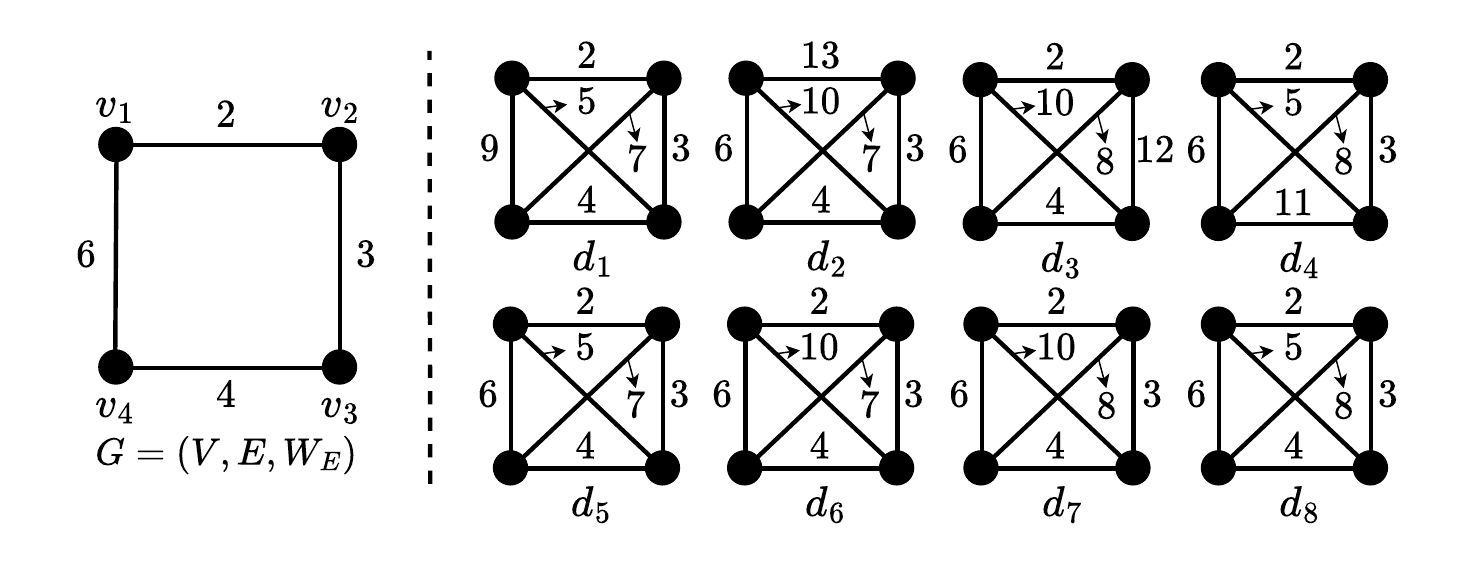} 
    \caption{The weighted graph 
    $G=(V,E,W_E)$ and all its possible path-representable distances.}
    \label{fig:Example_path_repre_distance}
\end{figure}

For the connected weighted graph \(G = (V, E, W_E)\) shown in the left figure of~\cref{fig:Example_path_repre_distance}, 
all path-representable distances derived from all the path choice functions are illustrated in the right figure of~\cref{fig:Example_path_repre_distance}, with a total of eight possibilities.
In this case, all the distances are weight-dominated, while only the distances \(d_5, d_6, d_7, d_8\) are cost-dominated.

\subsection{Main theorem}

Similarly, as discussed in~\cref{Thm: main theorem} with respect to $d_{\text{weight}}$ and $d_{\text{edge}}$,
the following generalization can be proved for cost-dominated path-representable distances:

\begin{restatable}{theorem}{mytheorem}
\label{Thm: main theorem_general}
Let $G=(V, E, W_E)$ be a connected weighted graph. 
Consider two cost-dominated path-representable distances $d_i$ and $d_j$ by path choice functions $g_i$ and $g_j$ on $G$, respectively. 
If $d_i(v,w) \le d_j(v,w)$ for any vertices $v$ and $w$, then there exists an injective function $\varphi_{i,j} : \textbf{bcd}_1(d_i) \rightarrow \textbf{bcd}_1(d_j)$ defined as $\varphi_{i,j}([\beta,\delta])=[\beta,\delta']$ such that $\delta \le \delta'$ for any $[\beta,\delta] \in \textbf{bcd}_1(d_i)$.   
\end{restatable}

\begin{figure}[ht]
    \centering
    \includegraphics[width=0.8\textwidth]{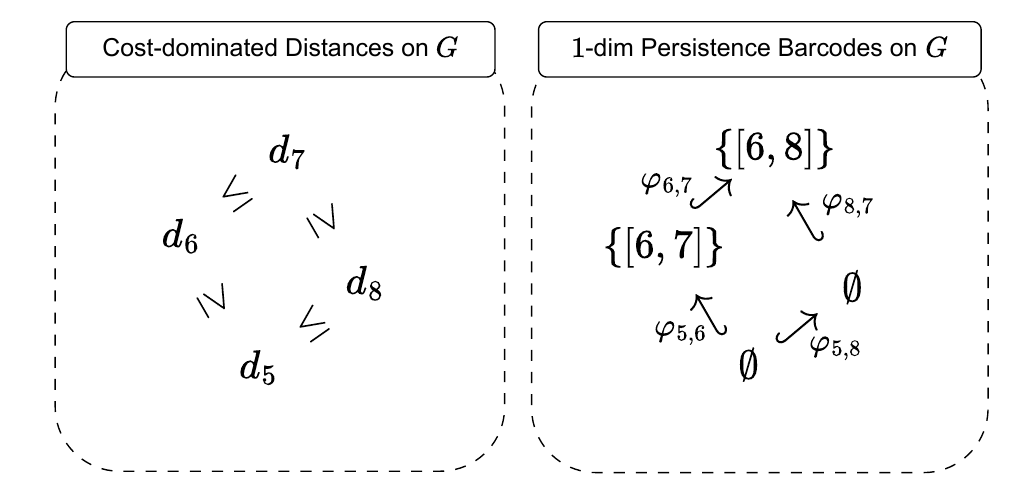} 
    \caption{
    The injective functions \(\varphi_{i,j} : \textbf{bcd}_1(d_i) \rightarrow \textbf{bcd}_1(d_j)\), as described in~\cref{Thm: main theorem_general}, for the four cost-dominated path-representable distances $d_5, d_6, d_7, d_8$ in~\cref{fig:Example_path_repre_distance}.
    }
    \label{fig:main_theorem_example_toy}
\end{figure}

\cref{fig:main_theorem_example_toy} illustrates the application of~\cref{Thm: main theorem_general} to the cost-dominated path-representable distances \(d_5, d_6, d_7, d_8\) described in~\cref{fig:Example_path_repre_distance}.
The left figure of~\cref{fig:main_theorem_example_toy} shows the partial order relationships among distances \(d_5, d_6, d_7, d_8\),
while the right figure displays the injective functions \(\varphi_{5,6}, \varphi_{6,7}, \varphi_{5,8}, \varphi_{8,7}\) between their 1-dimensional persistence barcodes.

\subsection{Remarks}

The main theorem can also be formulated using the concept of a poset. Let \(G\) be a connected weighted graph. To analyze the graph by defining the distance between the nodes, it is natural to define the set whose elements are the distances defined on \(G\). Especially in our theory, we focus on the poset \((\mathcal{L}, \le)\) whose elements are the cost-dominated path-representable distances defined on \(G\), where for two elements \(d_1, d_2 \in \mathcal{L}\), \(d_1 \le d_2\) if \(d_1(v, w) \le d_2(v,w)\) for any two points \(v, w \in V\). Furthermore, we calculate persistent homology using the standard algorithm~\cite{EHarer}, which calculates a persistence barcode from the given distance matrix. Note that \( (G,d) \) gives a unique distance matrix.  
We define such a function as \(\textbf{bcd}_k\), that is, \(\textbf{bcd}_k(d)\) is a \(k\)-dimensional persistence barcode obtained by \((G,d)\). Then the codomain of the function will be the poset of persistence barcodes on \(G\), as a subset of the poset of interval sets paired with \(k\)-simplices. Here, the partial order holds if we can find the proper inclusion between the elements with the same associated simplex in the sense of the main theorem. To be precise, for two $k$-dimensional persistence barcodes $S_1$ and $S_2$, we define \(S_1 \le S_2\) if there is an injection \(\varphi: S_1 \hookrightarrow S_2\) such that  
\(
\varphi([\beta,\delta]_{\sigma}) = [\beta,\delta']_{\sigma} \text{ for any }[\beta,\delta]_{\sigma} \in S_1 \text{ with } \delta \le \delta'. 
\) 
The main theorem claims that \(\textbf{bcd}_1\) becomes the order-preserving map.

Note that fixing the vertex order at least once is necessary to identify the exact injection. However, once the injection is identified, it will not change with respect to the change of vertex order. Hence, the whole theory holds when the codomain is a poset of persistence barcodes on \(G\) as a subset of the poset of multisets of intervals, carrying the injections defined with the associated birth edge. This gives canonicality to the theory.  
\begin{corollary}
    Consider the connected weighted graph \(G=(V,E,W_E)\).
    Let \(\mathcal{L}\) be the poset of path-representable distances from cost-dominated path choice functions on \(G\) and \(B\) be the poset of persistence barcodes on \(G\).  
    There is an order-preserving map \(\textbf{bcd}_1(\cdot)\) from \(\mathcal{L}\) to \(B\).   
\end{corollary}  

Considering the domain and range, it is natural to question if the least and the greatest elements exist.  
If we take the domain to be the set of cost-dominated path-representable distances on \(G\), then \(\textbf{bcd}_1\) becomes the order-preserving map. Hence if the least and greatest elements exist, they will map to the least and the greatest elements in the range.  
\begin{figure}[!ht]
    \centering
    \includegraphics[width=0.4\textwidth]{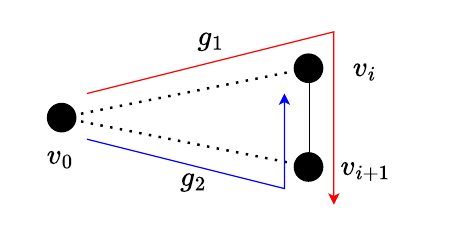} 
    \caption{The counterexample for the greatest element in \(\mathcal{L}\) illustrated in~\cref{prop:nogreatest}. Here \(g_1\) and \(g_2\) are the paths corresponding to \(d_1\) and \(d_2\), respectively.}
    \label{fig:Counterexample_Greatest}
\end{figure}
\begin{proposition}\label{prop:nogreatest}
    There is always the least element in the set of cost-dominated path-representable distances on \(G\), while there is no greatest element in general.
\end{proposition}
\begin{proof}
    First, if we choose \(d_{\text{weight}}\) as an element, then it becomes the least element according to~\cref{def: distance_weight}. Note that \(d_{\text{weight}}\) is indeed the element in the domain by~\cref{dweight_dedge_elements}. \\  
    Now we take a counterexample for the greatest element. 
    Let \(G\) be a cycle graph with at least 5 vertices and suppose that \(d_{\max}\) is the greatest element in \(P\).  
    As illustrated in~\cref{fig:Counterexample_Greatest}, fix a control point \(v_{0}\). Since \(d_{\max}\) is path-representable, for every other vertex \(v\) of \(G\), we must choose between two paths from \(v_{0}\) to \(v\) which \(d_{\max}\) calculates.  
    Let \(v_{i}\) and \(v_{i+1}\) be two consecutive nodes. Let \(d_1\) be the distance that chooses the path from \(v_{0}\) to \(v_{i}\) that does not contain \(v_{i+1}\) and the path from \(v_{0}\) to \(v_{i+1}\) that contains \(v_{i}\). Let \(d_2\) be the distance that chooses the path from \(v_{0}\) to \(v_{i+1}\) that does not contain \(v_{i}\) and the path from \(v_{0}\) to \(v_{i}\) that contains \(v_{i+1}\).  
    Note that we can make these two distances cost-dominated since \(v_{i}\) and \(v_{i+1}\) are not adjacent to \(v_{0}\), and \(G\) is a cycle graph.  
    Then \(d_1(v_{0}, v_{i}) > d_2(v_{0}, v_{i})\) and \(d_1(v_{0}, v_{i+1}) < d_2(v_{0}, v_{i+1})\). Thus, due to the assumption for \(d_{\max}\), we have \(d_{\max}(v_{0}, v_{i}) = d_1(v_{0}, v_{i})\) and \(d_{\max}(v_{0}, v_{i+1}) = d_2(v_{0}, v_{i+1})\) since there are only two candidates for each selection of path.  
    This leads to a contradiction to the consistency condition of path-representable distance.
\end{proof}

\begin{figure}[ht]
    \centering
    \includegraphics[width=0.9\textwidth]{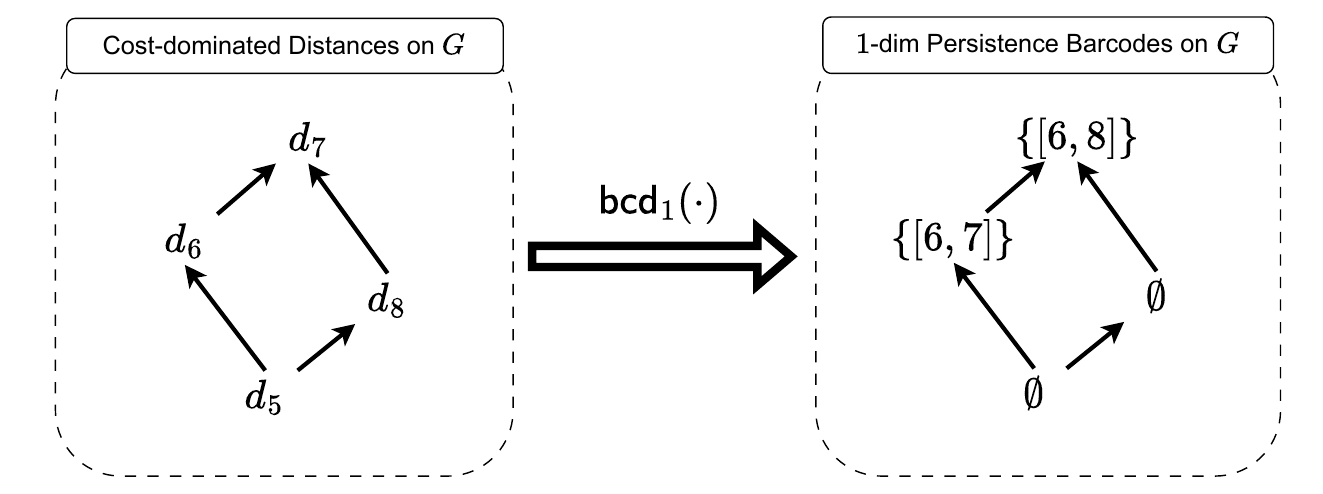} 
    \caption{\(\textbf{bcd}_1\) being an order preserving map between the poset of cost-dominated path-representable distance on $G$~(Left) and the poset of $1$-dimensional persistence barcode on $G$ from~\cref{fig:Example_path_repre_distance}~(Right). Here the arrow denotes the partial order for each poset.} 
    \label{fig:Example_poset2}
\end{figure}
\section{Proofs of main theory}\label{sec:Proofs of main theory}

\begin{lemma}\label{lemma1}
For a weighted graph $G=(V,E,W_E)$, let $d:V \times V \rightarrow \mathbb{R}$ be any distance defined on $G$. 
 If $\sigma=(v,w)$ is a birth simplex of dimension $1$ with the distance $d(v,w)$ of a $1$-dimensional cycle, there is no vertex $v_0 \in V(G)$ such that $d(v,v_0)<d(v,w)$ and $d(w,v_0)<d(v,w)$.
\end{lemma}

\begin{proof}
Let $e=(v,w)$ be a birth edge with the distance $d(v,w)$ of a $1$-dimensional cycle. 
Suppose that there is a vertex $v_0 \in V(G)$ such that $d(v,v_0)<d(v,w)$ and $d(w,v_0)<d(v,w)$.
If we denote $t_0 = \max(d(v,v_0),d(w,v_0))$, then the edges $(v,v_0), (w,v_0)$ are already in the simplicial complex $\mathbb{X}_{t_0}$.
Now, considering the $\mathbf{R}=\mathbf{D}\mathbf{V}$ decomposition, with the boundary matrix $D$, the $2$-simplex $( v, w, v_0)$ maps to $(v,w) + (v,v_0) + (w,v_0)$. 
Of those three edges, $(v,w)$ is the most recent edge born before $(v,w,v_0)$. Thus $(v,w)$ is necessarily chosen as the index of $\text{low}_{\mathbf{R}}$.
The edge $e =(v,w) $ has its birth value as $d(v,w)$ and its death value as $d(v,w)$ because $(v,w)$ and $(v,w,v_0)$ are created simultaneously.
This contradicts the fact that $e$ is a birth edge.
\end{proof}

\begin{proposition}\label{prop1}
  Let $G=(V,E,W_E)$ be a connected weighted graph.
  Consider any path choice function $g$ on $G$ and its path-representable distances $d$ along $g$.
  Then every birth simplex of dimension $1$ of $\textbf{bcd}_1(d)$ exists in $G$.
  Therefore, it is justified to call a birth simplex of dimension $1$ as a birth `edge'.
\end{proposition}

\begin{proof}
  Let $e=(v,w)$ be a birth simplex of dimension $1$ with a birth value $d(v,w)$. Suppose $ e \notin E$.
  There is a path $g(v,w)$ in $G$ such that $d(v,w) = \sum_{e \in g(v,w)} W_E(e)$.
  Denote the path $g(v,w)$ as the sequence of vertices $v=v_0, v_1, v_2 , \ldots , v_{m-1}, v_m = w$ and $a_k=W_E(( v_{k-1},v_k ))$.
  Then, by the consistency property of the path choice function $g$, the path $g(v, w_{m-1})$ is a sequence of $v_0$, $v_1$, $v_2$, $\ldots$, $v_{m-1}$. 
  Therefore, we know that $d(v, v_{m-1}) = a_1 + a_2 + \cdots + a_{m-1} < d(v,w)$.
  Since $d(v_{m-1},w) < d(v,w)$ is obvious, this contradicts~\cref{lemma1}.
 \end{proof}

Let $\mathcal{B}$ be the set of all birth edges in $\textbf{bcd}_1(d)$.
By~\cref{prop1}, we know that $\mathcal{B} \subseteq E$.
Generally, for any edge $e = (v,w) \in E$, it is not necessary that $W_E(e) = d(v,w)$. For example, cases where $W_E(e)$ $> d_{\text{weight}}(v,w)$ are common.
However, for the birth edges, the following interesting proposition holds.
\begin{proposition}\label{prop:birth_edge_weigh_fix}
  For any $e=(v,w) \in \mathcal{B}$, we always have $W_E(e) = d(v,w)$.
\end{proposition}

\begin{proof}
Consider any birth edge $e = (v,w) \in \mathcal{B}$. 
Suppose that $W_E(e) \neq d(v,w)$. Then the path $g(v,w)$ is not the edge $e$. 
Write $g(v,w)$ as $v=v_0$, $v_1$, $\ldots$ , $v_m = w$. 
We know that $d(v,w) = \sum_{e' \in g(v,w)} W_E(e') = \sum_{k=1}^{m}W_E((v_{k-1},v_k ))$.
Then $d(v,v_{m-1}) =  \sum_{k=1}^{m-1}W_E((v_{k-1},v_k )) < d(v,w)$ because the path $g(v_0,v_{m-1})$ is a sequence $v_0$, $v_1$, $\ldots$, $v_{m-1}$ and $W_E$ is the positive function. 
Since $d(v_{m-1},w) < d(v,w)$, obviously this contradicts~\cref{lemma1}.  
\end{proof}

Consider all path-representable distances on the graph \( G = (V, E, W_E) \).
Then, by~\cref{prop1,prop:birth_edge_weigh_fix},
their birth simplices of dimension $1$ become elements of \( E \), and the distances between the endpoints of the edges are given by the weight \( W_E \).

Let $G=(V,E,W_E)$ be any weighted graph. 
A \textit{tree} of $G$ is a connected subgraph which does not have a cycle. 
A \textit{forest} of $G$ is a subgraph which does not have a cycle.
A \textit{spanning tree} of $G$ is a tree of $G$ that includes all the vertices of $G$.
For fixed $G$, there are many spanning trees of $G$.
A \textit{minimum spanning tree} of $G$ is a spanning tree of $G$ that has the smallest sum of weights among all possible spanning trees of $G$.
Let $\text{MST}(G)$ denote the set of all minimum spanning trees of $G$.

\begin{lemma}\label{lemma:cycle_property_tree}
Let $G=(V,E,W_E)$ be any connected weighted graph and $\mathcal{T} \in \text{MST}(G)$.
Suppose that we have a cycle $\tau$ in $G$ and that there is an edge $e \in \tau$ such that $W_E(e) > W_E(e')$ for any other edge $e' \in \tau$.
Then $e \notin \mathcal{T}$. 
\end{lemma}
\begin{proof}
    Assume that $e \in \mathcal{T}$. 
    As $\mathcal{T}$ is a tree, if we remove \(e\), then $\mathcal{T}$ is split into two subtrees $\mathcal{T}_1$ and $\mathcal{T}_2$. 
    Note that \(\tau - e\) is also a path that connects $\mathcal{T}_1$ and $\mathcal{T}_2$. 
    Because $\mathcal{T}$ is a spanning tree, every vertex in \(\tau - e\) is in $\mathcal{T}_1$ or $\mathcal{T}_2$.
    Therefore, there exists an edge \(e'\) in \(\tau\) connecting $\mathcal{T}_1$ and $\mathcal{T}_2$ such that $W_E(e) > W_E(e')$.
    However, replacing the edge \(e\) in \(\mathcal{T}\) with \(e'\) results in a tree \(\mathcal{T}_1 \cup \mathcal{T}_2 \cup \{ e' \} \) that is a spanning tree with a sum of weights smaller than \(\mathcal{T}\). This is a contradiction.
\end{proof}

\begin{definition}[Edge swap process of spanning trees]
    Let $\mathcal{T}$ be a spanning tree in a connected graph $G$. 
    Consider any edge $e \notin \mathcal{T}$. Then there exists a unique path $p$ in $\mathcal{T}$ connecting the endpoints of $e$. 
    For any edge $e_p$ in the path $p$, we can replace the edge $e$ with $e_p$. 
    This results in another spanning tree $\mathcal{T}' = \mathcal{T} \cup \{e\} - \{e_p\}$.
    This operation is called the \textit{edge swap process} of a spanning tree.
\end{definition}

Consider a connected weighted graph $G=(V,E,W_E)$.
Define $\bar{E} = \{ (v,w) \mid v,w \in V \}$.
For any path choice function $g$, we have its path-representable distance $d_g$ and a weighted complete graph $K^g = (V,\bar{E},W_{\bar{E}})$, where $W_{\bar{E}}(e) = d_g(v,w)$ for all $e=(v,w) \in \bar{E}$.
We name this graph $K^g$ \textit{the graph completion of $G$ by $g$}.

\begin{proposition}\label{prop:MST_same}
Let $G=(V,E,W_E)$ be a connected weighted graph, and let $K^g=(V,\bar{E},W_{\bar{E}})$ be its graph completion by $g$ as explained. 
If $g$ is weight-dominated, then we have
$\text{MST}(K^g) = \text{MST}(G)$ as the weighted graph.
\end{proposition}

\begin{proof}
$(\text{MST}(K^g) \subseteq \text{MST}(G))$. 
(i) First, we show that any minimum spanning tree of ${K^g}$, denoted by $\mathcal{T}$, is a subgraph of $G$ so it is also a spanning tree of $G$.
Suppose that there is an edge $e=(v,w) \in \bar{E}$ in $\mathcal{T} \setminus E$. 
Then $W_{\bar{E}}(e) = d(v,w)=\sum_{e' \in g(v,w)}W_E(e')$, where $g(v,w)$ is different from $e$ because $e$ is not in $G$.
Observe that $\tau=g(v,w)+e$ is a cycle.
By the consistency of $g$, for any edge $e'=(v_0,w_0)$ in $g(v,w)$ we know the path $g(v_0,w_0)$ is just the edge $(v_0,w_0)$, so $W_{\bar{E}}(e') = W_E(e')$.
This contradicts~\cref{lemma:cycle_property_tree} because we have
\(
    W_{\bar{E}}(e) = \sum_{e' \in g(v,w)}W_E(e') = \sum_{e' \in g(v,w)}W_{\bar{E}}(e') > W_{\bar{E}}(e')
\)
and $\tau$ is a cycle in $K^g$ with $e \in \mathcal{T}$.

(ii) Second, we show that $W_E = W_{\bar{E}}$ on $\mathcal{T}$.
Consider any edge $e=(v,w) \in \mathcal{T}$.
Suppose that $W_E(e) \neq  W_{\bar{E}}(e)$.
This means that there is a path $g(v,w) \neq e$ in $G$.
Observe that $\tau=g(v,w)+e$ is a cycle.
This contradicts~\cref{lemma:cycle_property_tree} 
for the same reason as in (i) because $\mathcal{T}$ is a minimum spanning tree of $K^g$.

(iii) Consider a minimum spanning tree $\mathcal{T}' \in \text{MST}(G)$. If $\mathcal{T}=\mathcal{T}'$, then the proposition holds. Suppose that $\mathcal{T} \neq \mathcal{T}'$, which means that there is an edge $e \in \mathcal{T} \setminus \mathcal{T}'$ and a unique path $p'$ in $\mathcal{T}'$ connecting the endpoints of $e$.
In the path $p'$, there is an edge $e' \in \mathcal{T}' \setminus \mathcal{T}$ because $e + p'$ is a cycle in $G$ and $\mathcal{T}$ is a tree.
By $(ii)$, we have $W_E(e) = W_{\bar{E}}(e)$.
Suppose that $W_E(e') \neq W_{\bar{E}}(e')$ and denote $e'=(v,w)$.
Then there is a path $g(v,w)$ such that $W_E(e') > W_E(e_i)$ for any $e_i \in g(v,w)$.
This contradicts~\cref{lemma:cycle_property_tree} because $e' \in \mathcal{T}' \in \text{MST}(G)$ so that $W_E(e') = W_{\bar{E}}(e')$.
Since $\mathcal{T}' \in \text{MST}(G)$, we have $W_E(e') \leq W_E(e)$. Otherwise, by replacing $e'$ with $e$ in the edge swap process, we have a contradiction as $\mathcal{T}' \in \text{MST}(G)$.
Similarly, since $\mathcal{T} \in \text{MST}(K^g)$, we have $W_{\bar{E}}(e') \geq W_{\bar{E}}(e)$.
To summarize, the following hold:
\begin{enumerate}
    \item $W_E(e) = W_{\bar{E}}(e)$ and $W_E(e') = W_{\bar{E}}(e')$.
    \item $W_E(e') \leq W_E(e)$ and $W_{\bar{E}}(e') \geq W_{\bar{E}}(e)$.
\end{enumerate}
Therefore, we conclude that $W_E(e) = W_E(e')$. If we now consider $\mathcal{T} \cup \{ e'\} - \{ e \}$ as a new $\mathcal{T}$, we can repeat the above process. Since this process will terminate after a finite number of steps, we ultimately achieve $\mathcal{T} = \mathcal{T}'$, thereby completing the proof.

$(\text{MST}(G) \subseteq \text{MST}(K^g))$. The proof can be demonstrated by almost the same method.
\end{proof}

\cref{prop:MST_same} implies that any weight-dominated path choice function $g$ always keeps $\text{MST}(G)$ invariant, for which the weight-dominated condition is essential. 
Otherwise, neither $\text{MST}(G)$ nor $\text{MST}(K^g)$ can be guaranteed to be larger. 
A counterexample is shown in~\cref{fig:counter_example} below.
\begin{figure}[h]
    \centering
    \includegraphics[width=0.9\textwidth]{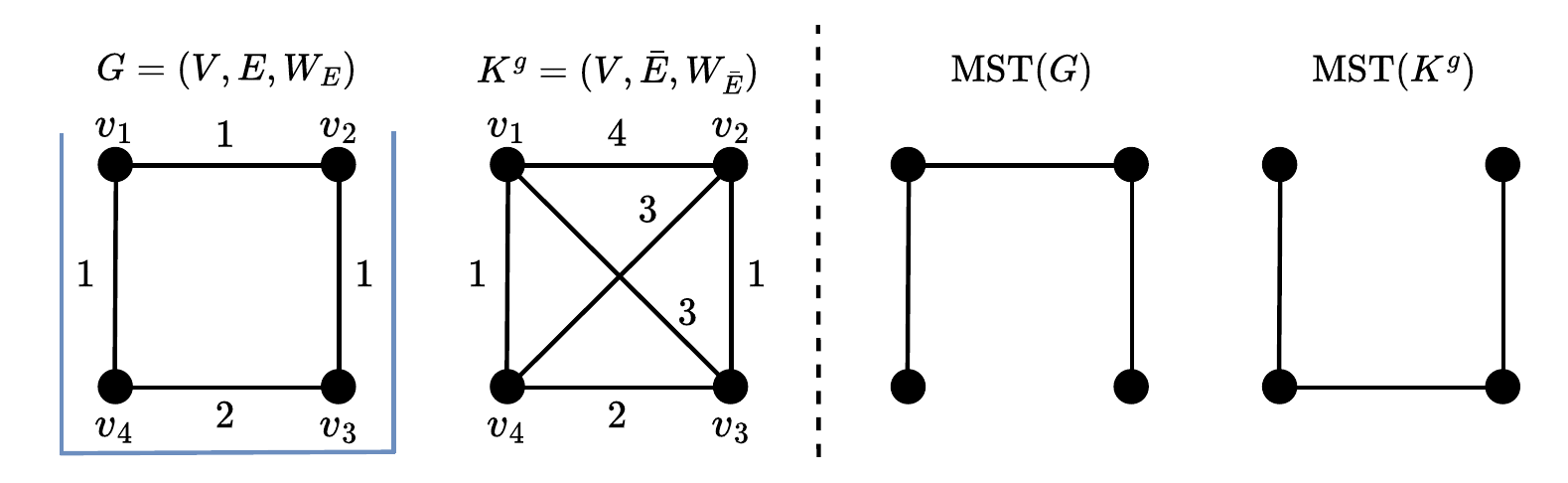} 
    \caption{A counterexample demonstrating that without the weight-dominated condition, neither $\text{MST}(G)$ nor $\text{MST}(K^g)$ can be guaranteed to be larger.}
    \label{fig:counter_example}
\end{figure}

In fact, the construction of a minimum spanning tree naturally arises from the standard algorithm discussed in~\cref{subsec:Persistent homology}.
\begin{lemma}\label{lemma:MST_standard_algorithm} 
 For the fixed order $\pi$ on the edge set $\bar{E}$ in $K^g$, consider the set of all $1$-simplices $\sigma_j$ such that $i = \text{low}_{\mathbf{R}}(j)$ for some $0$-simplex $\sigma_i$, denoted by $\mathcal{T}^{\pi}$.
 Then \[
  \textnormal{MST}(K^g) = \{ \mathcal{T}^{\pi} \mid  \pi \text{ is an order on edge set $\bar{E}$ in $K^g$}\}
 \]
\end{lemma}
\begin{proof}
    $(\supseteq)$
    The process of sequentially constructing a minimum spanning tree matches precisely the process of incrementally calculating \(\text{low}_{\mathbf{R}}\). 
    To explain in more detail, for every vertex in $K^g$, each $0$-simplex (vertex) is connected to a $1$-simplex (edge) causing its death. According to the definition of $\text{low}_{\mathbf{R}}$, no cycles can form hence it is a spanning tree. Since $\text{low}_{\mathbf{R}}$ selects edges in order from the lowest value, it becomes the minimum spanning tree.
    Therefore, for any order $\pi$ on the edge set $\bar{E}$ in $K^g$, $\mathcal{T}^{\pi}$ is an element in $\text{MST}(K^g)$.

    $(\subseteq)$
    Consider an arbitrary element $\mathcal{T} \in \text{MST}(K^g)$. Let $\pi$ be an order on $\bar{E}$ such that, for any given weights on the edges, if the weights are the same, the edges in $\mathcal{T}$ are preferred over those in $K^g \setminus \mathcal{T}$. 
    When computing the persistent homology of $K^g$, even if the order of the $1$-simplices (edges) corresponding to the indices of the matrix in the $\mathbf{R} = \mathbf{D}\mathbf{V}$ algorithm is rearranged according to $\pi$, the result of the persistent homology computation remains unchanged (permutation invariant). In this case, $\mathcal{T} = \mathcal{T}^{\pi}$.
\end{proof}

\begin{proposition}\label{prop:MST_standard_algorithm_same}
    Let $G=(V,E,W_E)$ be a connected weighted graph. Consider any weight-dominated choice functions $g_1$ and $g_2$ and their completions $K^{g_1}$ and $K^{g_2}$.
    Then there exists an order $\pi$ in $\bar{E}$ such that their minimum spanning trees on $K^{g_1}$ and $K^{g_2}$ induced by the standard algorithm are exactly the same. 
\end{proposition}
\begin{proof}
    First, since $g_1$ and $g_2$ are weight-dominated, by~\cref{prop:MST_same}, we have $\text{MST}(K^{g_1}) = \text{MST}(K^{g_2}) = \text{MST}(G)$. 
    Now, let the minimum spanning trees obtained by the standard algorithm for $K^{g_1}$ and $K^{g_2}$ be denoted as $\mathcal{T}_1$ and $\mathcal{T}_2$, respectively. 
    Then, by~\cref{lemma:MST_standard_algorithm}, there exists an order $\pi$ on $\bar{E}$ such that $\mathcal{T}_1 = \mathcal{T}_2^{\pi}$.
\end{proof}

\begin{proposition}\label{prop:birth_edge_inclusion}
Let $G=(V,E,W_E)$ be a connected weighted graph. 
Consider any two path-representable distances $d_1$ and $d_2$ from two cost-dominated path choice functions $g_1$ and $g_2$.
Let $\mathcal{B}_1$ and $\mathcal{B}_2$ be the set of all birth edges of $\textbf{bcd}_1(d_1)$ and $\textbf{bcd}_1(d_2)$.
If $d_2(v,w) \le d_1(v,w)$ for any vertices $v$ and $w$, then $\mathcal{B}_2 \subseteq \mathcal{B}_1$ for some simplex ordering $\pi$ on $\mathbb{X}=\mathcal{P}(V)$.
\end{proposition}

\begin{proof} 
  For a birth edge $e = \sigma_i \in \mathcal{B}_2$, suppose that we have a pair of simplices $(\sigma_i,\sigma_k)$ such that $\text{low}_{\mathbf{R}_2}(k)=i$ for some $\sigma_k$, where $\mathbf{R}_2=\mathbf{D}_2\mathbf{V}_2$ in the calculation of $\textbf{bcd}_1(d_2)$.
  Let $f_2$ be a real-valued function induced by $d_2$ as in~\cref{subsec:Persistent homology}. 
  Since $e$ is a birth edge, $f_2(\sigma_i) < f_2(\sigma_k)$.

  (i) First, we show that the $e$-column of the reduced $R_1$ is zero.
  From $e \in \mathcal{B}_2$, we know that the $e$-column of $\mathbf{R}_2$ is zero.
  The fact that the $e$-column of $\mathbf{R}_2$ is a zero column means that there are column indices $j_1$, $\ldots$, $j_m$ corresponding to edges $e_{j_1}, \ldots, e_{j_m}$ such that $e + \sum_{l=1}^m e_{j_l}$ is a cycle and there are $0$-simplices $\sigma_{i_{j_k}}$ such that $i_{j_k} =\text{low}_{\mathbf{R}_2}(j_k)$.
  By~\cref{prop:MST_standard_algorithm_same}, we can assume that the minimum spanning trees $\mathcal{T}_1$ and $\mathcal{T}_2$ for $\mathbf{R}_1$ and $\mathbf{R}_2$ in the standard algorithm are exactly the same for some fixed order $\pi$ on edge set $\bar{E}$.
  Therefore, we know that $i_{j_k} = \text{low}_{\mathbf{R}_1}(j_k)$ for all $k$. Thus the $e$-column of $\mathbf{R}_1$ is zero due to the column indices $j_1, \ldots, j_m$ during the column operations in the standard algorithm.

  (ii) Next, we show that if the $1$-cycle $\tau$ corresponding to the birth edge $e=\sigma_i$ is terminated by the $2$-simplex $\sigma_j$ in $\mathbf{R}_1$, then $f_1(\sigma_i) < f_1(\sigma_j)$.
  Consider a pair of simplices  $(\sigma_i,\sigma_j)$ which is $\text{low}_{\mathbf{R}_1}(j)=i$.
  For any $t \in \mathbb{R}$, define $\mathbb{X}^2_t = {f_2}^{-1}((-\infty,t])$ and $\mathbb{X}^1_t = {f_1}^{-1}((-\infty,t])$.
  As $f_2 \le f_1 $ on $\mathbb{X}$, we know that $\mathbb{X}^1_t \subseteq \mathbb{X}^2_t $ for any $t$.
  In filtration on $\mathbb{X}^1$, the $1$-cycle $\tau$ is terminated by $\sigma_j$ at a death value $f_1(\sigma_j)$. 
  That is, $\tau$ is zero homologous in $\mathbb{X}^1_{f_1(\sigma_j)}$.
  However, since $\mathbb{X}^1_{f_1(\sigma_j)} \subseteq \mathbb{X}^2_{f_1(\sigma_j)}$, $\tau$ is also terminated at $f_1(\sigma_j)$ in filtration $\mathbb{X}^2$. 
  Therefore, we obtain $f_2(\sigma_k) \le f_1(\sigma_j)$ because the value $f_2(\sigma_k)$ is a death value of $\tau$, which means the first value to die in filtration $\mathbb{X}^2$. 
  
  Now, by~\cref{prop:birth_edge_weigh_fix}, we know that $W_E(e) = W_{\bar{E_2}}(e) = f_2(e)$ since $e \in \mathcal{B}_2$. 
  Then, since $g_1$ is cost-dominated, we have  \(f_1(e)=W_{\bar{E}_1}(e) \le W_E(e)=f_2(e)\).
  By $d_2 \leq d_1$, we know that $f_2(e) \le f_1(e)$ and thus $f_2(e)=f_1(e)$.
  Since $e \in \mathcal{B}_2$, we have $f_2(e) < f_2(\sigma_k)$.
  Therefore, we obtain $f_1(e)= f_2(e) < f_2(\sigma_k) \le f_1(\sigma_j) $. 

  Finally, by (i) and (ii), an edge $e$ is a birth edge in $\textbf{bcd}_1(d_1)$. This completes the proof.
\end{proof}

The following is the proof of the main theorem.
\mytheorem*

\begin{proof}
    We show that we can construct an injection from $\textbf{bcd}_1(d_2)$ to $\textbf{bcd}_1(d_1)$.
    By~\cref{prop:MST_standard_algorithm_same}, we can take the order $\pi$ in the edge set $\bar{E}$ so that the minimum spanning trees $\mathcal{T}_1$ and $\mathcal{T}_2$ for $K^{g_1}$ and $K^{g_2}$ in the standard algorithm are exactly the same.
    For fixed order $\pi$, we can express any points in $\textbf{bcd}_1(d_2)$ as $[\beta,\delta]_{\sigma_i}$ corresponding to $(\sigma_i,\sigma_k)$ with $i =\text{low}_{\mathbf{R}_2}(k)$.
    Here, $\sigma_i$ is a birth simplex of dimension $1$ for the point $[\beta,\delta]$.
    By~\cref{prop:birth_edge_inclusion}, $\sigma_i$ is also a birth simplex of dimension $1$ in $\textbf{bcd}_1(d_1)$, so there is a barcode $[\beta',\delta']_{\sigma_i}$ corresponding to $(\sigma_i,\sigma_j)$ such that $i = \text{low}_{\mathbf{R}_1}(j)$.
    We know that $\beta=\beta'$ because by~\cref{prop:birth_edge_weigh_fix}, we have 
    \[
     \beta=f_2(\sigma_i)=d_2(v,w) = W_E(\sigma_i)= d_1(v,w) = f_1(\sigma_i) =\beta',
    \]
    where $\sigma_i = (v,w)$.
    Also, in (ii) in the proof of~\cref{prop:birth_edge_inclusion}, we already show that $\delta = f_2(\sigma_k) \le f_1(\sigma_j) = \delta'$. 
    Now, define a function $\varphi$ as $\varphi([\beta, \delta]_{\sigma_i}) = [\beta,\delta']_{\sigma_i}$ for any $[\beta, \delta] \in \textbf{bcd}_1(d_2)$.
    This map is a well-defined injective function because every point in $\textbf{bcd}_1$ is uniquely determined by the birth simplex $\sigma_i$ of dimension $1$.
    This completes the proof.
\end{proof}
\section{Analysis of dimensional variations of the main theorem}\label{sec:Analysis of dimensional variations of the main theorem}

In this section, we briefly discuss whether~\cref{Thm: main theorem_general} can be extended to arbitrary $k$-dimensional persistence barcodes. 
Will it hold for the higher dimensions? Generally, no.
Although we were unable to provide a proof for the general case of $k$, 
we have found counterexamples experimentally demonstrating that the theorem does not hold for \( 2 \le k \le 5\).

\begin{statement}\label{statement: high_dimension}
  Let $G=(V,E,W_E)$ be a connected weighted graph. 
  Consider two cost-dominated path choice functions $g_1$ and $g_2$ on $G$, and their two path-representable distances $d_1$ and $d_2$ along $g_1$ and $g_2$.
  For any dimension $k$ if $d_2(v,w) \le d_1(v,w)$ for any vertices $v,w$, then there exists an injective function $\varphi : \textbf{bcd}_k(d_2) \rightarrow \textbf{bcd}_k(d_1)$.
\end{statement}

\begin{table}[!ht]
\centering
\caption{True or False for high dimensional cases \(k\) for~\cref{statement: high_dimension}}
\begin{tabular}{l|c|c|c|c|c|c|}
\hline
\textbf{Dimension k}& \textbf{0}& \textbf{1}& \textbf{2} & \textbf{3} & \textbf{4} & \textbf{5}\\
\hline
\textbf{(T)rue or (F)alse} & T & T & F & F & F & F \\
\hline
\end{tabular}
\label{tab:detect_ratio}
\begin{minipage}{0.9\textwidth}
    \footnotesize
    \vspace{2mm} 
    Note: This table shows the results of~\cref{statement: high_dimension} for different dimensional cases. The table indicates whether the theorem holds (True) or does not hold (False) for each dimension.
\end{minipage}
\end{table}

In the case where \( k = 0 \) in~\cref{statement: high_dimension}, \(\textbf{bcd}_k(d_{\text{edge}})\) and \(\textbf{bcd}_k(d_{\text{weight}})\) become exactly the same, therefore satisfying the condition.
To prove that~\cref{statement: high_dimension} is false for \( 2 \le k \le 5\), we performed experiments using Python code. Specifically, knowing that \(d_{\text{edge}} \geq d_{\text{weight}}\), we sought to identify the weighted graphs \(G = (V, E, W_E)\) where the condition \(\lvert \textbf{bcd}_k(d_{\text{edge}}) \rvert < \lvert \textbf{bcd}_k(d_{\text{weight}}) \rvert\) is satisfied for dimension \(k\). All codes used in this study are available at \url{https://github.com/AI-hew-math/embedding}.

We conjecture that \cref{statement: high_dimension} is false for any dimension $k \geq 2$ and that a generalized definition of distance with the extended weight, and cost-dominated concepts might exist for higher dimensions, leading to a similar injection relationship as demonstrated for the cases $k=0, 1$.
\section{Total persistence difference and stability theorem}
\label{sec:Data Analysis}

Let $G=(V,E,W_E)$ be a connected weighted graph.
Given a path–representable distance $d$ on $V$, the associated
Vietoris–Rips filtration yields a $1$–dimensional persistence barcode
$\textbf{bcd}_1(G,d)$, which is a multiset of intervals
$[\beta,\delta] \subset \mathbb{R}$.
In this section we compare two such barcodes obtained from
two comparable distances on the same graph and introduce a quantity
that measures how the $1$–dimensional structure changes
as the distance varies.

Suppose $d_i \le d_j$ are cost–dominated path–representable distances on $G$.
By \cref{Thm: main theorem_general}, the $1$–dimensional barcodes
$\textbf{bcd}_1(G,d_i)$ and $\textbf{bcd}_1(G,d_j)$ can be indexed by
the birth edges in $E$:
each interval is associated with the edge that first creates the
corresponding nontrivial $1$–cycle.
If an edge $\sigma$ appears as a birth edge in both barcodes, then their birth and death values satisfy
\[
  \beta^i_\sigma = \beta^j_\sigma
  \quad\text{and}\quad
  \delta^i_\sigma \le \delta^j_\sigma.
\]
If a birth edge $\sigma$ appears only in $\textbf{bcd}_1(G,d_j)$,
then $\textbf{bcd}_1(G,d_j)$ contains an interval
$[\beta^j_\sigma,\delta^j_\sigma]$ for which there is no corresponding
interval in $\textbf{bcd}_1(G,d_i)$.
This edge–indexed viewpoint suggests a natural way to quantify the difference between $\textbf{bcd}_1(G,d_i)$ and $\textbf{bcd}_1(G,d_j)$.

\begin{definition}\label{def:tpd}
Let $G=(V,E,W_E)$ be a connected weighted graph and
$d_i \le d_j$ be cost–dominated path–representable distances on $G$.
Let $\mathcal{B}_i$ and $\mathcal{B}_j$ denote the sets of birth edges
in the $1$–dimensional barcodes $\textbf{bcd}_1(G,d_i)$
and $\textbf{bcd}_1(G,d_j)$, respectively.
For each birth edge $\sigma$ we write
$(\beta^i_\sigma,\delta^i_\sigma)$ and
$(\beta^j_\sigma,\delta^j_\sigma)$ for the birth and death values
in $\textbf{bcd}_1(G,d_i)$ and $\textbf{bcd}_1(G,d_j)$ whenever they exist.
For $p\ge1$ we define the \emph{total persistence difference} by
\[
\mathrm{Diff}_p\!\left(\textbf{bcd}_1(G,d_j),
                      \textbf{bcd}_1(G,d_i)\right)
=
\Bigg(
\sum_{\sigma\in\mathcal{B}_i}
(\delta_\sigma^j - \delta_\sigma^i)^p
+
\sum_{\sigma\in\mathcal{B}_j\setminus\mathcal{B}_i}
(\delta_\sigma^j - \beta_\sigma^j)^p
\Bigg)^{1/p}.
\]
\end{definition}

The first term sums up, in an $\ell^p$–fashion, the increase in death
values for those intervals that appear in both barcodes and are indexed
by the same birth edge $\sigma\in\mathcal{B}_i$.
The second term accounts for intervals that appear only in
$\textbf{bcd}_1(G,d_j)$, i.e.\ those whose birth edges lie in
$\mathcal{B}_j\setminus\mathcal{B}_i$, and measures their persistence
$\delta^j_\sigma - \beta^j_\sigma$.
Thus $\mathrm{Diff}_p$ quantifies how much the $1$–dimensional barcode extends when the associated distance increases from $d_i$ to $d_j$, using only the inclusion-based indexing provided by the graph.
A schematic illustration is shown in \cref{fig:tpd_demo}.

\begin{figure}[h]
    \centering
    \includegraphics[width=0.45\textwidth]{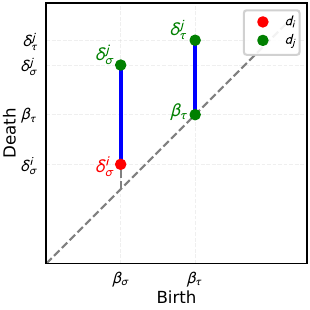}
    \caption{Illustration of the total persistence difference.
    The vertical blue segments represent the quantities used in
    \cref{def:tpd}.
    On the left, an interval with the same birth value
    $\beta_{\sigma}$ for $d_i$ (red) and $d_j$ (green) contributes
    $\delta^j_{\sigma}-\delta^i_{\sigma}$.
    On the right, an interval that appears only for $d_j$, with birth
    value $\beta^j_{\tau}$ and death value $\delta^j_{\tau}$, contributes
    $\delta^j_{\tau}-\beta^j_{\tau}$.
    The value $\mathrm{Diff}_p$ is obtained by taking the $\ell^p$-norm
    of all such contributions.}
    \label{fig:tpd_demo}
\end{figure}
Classical distances such as the bottleneck or Wasserstein distance compare persistence barcodes by selecting a matching between points in the two barcodes and optimizing a transport cost.
In the present setting, the barcodes are generated from the same graph equipped with two comparable distances, and there is a natural way to relate intervals via their common birth edges.
The quantity $\mathrm{Diff}_p$ is defined directly in terms of this inclusion-based correspondence.

\subsection{Stability Theorem}

We now provide a Lipschitz stability bound for $\mathrm{Diff}_p$ with respect to pointwise perturbations of the underlying distances.

\begin{theorem}\label{Thm:tpd-stability}
Let $G=(V,E,W_E)$ be a connected weighted graph and
$d_i \le d_j$ be cost–dominated path–representable distances on $G$.
Then
\[
\mathrm{Diff}_p\!\left(\textbf{bcd}_1(G,d_j),
                      \textbf{bcd}_1(G,d_i)\right)
\;\le\;
\gamma\, \lVert d_j - d_i \rVert_\infty,
\]
where
\[
\gamma = (|E|-|V|+1)^{1/p} \quad \text{and} \quad 
\lVert d_j - d_i \rVert_\infty
:= \max_{v,w\in V}\bigl(d_j(v,w) - d_i(v,w)\bigr).
\]
\end{theorem}

\begin{proof}
Let $f_i$ and $f_j$ be the filtration functions associated with
$d_i$ and $d_j$ as in \cref{subsec:Persistent homology}.
Consider the complete graphs obtained from these distances and fix
an ordering of edges so that the minimum spanning tree produced by the
standard algorithm is the same in both cases; this is possible by
\cref{prop:MST_standard_algorithm_same}.
Every nontrivial $1$–cycle then arises when an edge not belonging to
this tree is added, and birth edges are contained in $E$.

(i) Let $\sigma\in\mathcal{B}_i$.
By construction, the same edge $\sigma$ also becomes a birth edge
in $\textbf{bcd}_1(G,d_j)$, and the corresponding death values satisfy
$\delta^i_\sigma \le \delta^j_\sigma$.
If we had
\[
f_i^{-1}((-\infty,\delta^i_\sigma])
\;\subseteq\;
f_j^{-1}((-\infty,\delta^j_\sigma - \varepsilon])
\]
for some $\varepsilon>0$, then the $1$–cycle created by $\sigma$
would already be trivial at filtration value $\delta^j_\sigma-\varepsilon$
for $d_j$, contradicting the fact that its death value under $d_j$
is $\delta^j_\sigma$.
Hence there exists a pair $(v_\sigma,w_\sigma)\in V\times V$ such that
\[
d_i(v_\sigma,w_\sigma) \le \delta^i_\sigma
\quad\text{and}\quad
d_j(v_\sigma,w_\sigma) \ge \delta^j_\sigma.
\]
It follows that
\[
\delta^j_\sigma - \delta^i_\sigma
\;\le\;
d_j(v_\sigma,w_\sigma) - d_i(v_\sigma,w_\sigma)
\;\le\;
\lVert d_j - d_i \rVert_\infty.
\]

(ii) Now let $\sigma\in\mathcal{B}_j\setminus\mathcal{B}_i$.
In this case, the edge $\sigma$ creates a nontrivial $1$–cycle in the
filtration for $d_j$, but the corresponding cycle is trivial in the
filtration for $d_i$ at the same values.
If we had
\[
f_i^{-1}((-\infty,\beta^j_\sigma])
\;\subseteq\;
f_j^{-1}((-\infty,\delta^j_\sigma - \varepsilon])
\]
for some $\varepsilon>0$, then the cycle would already be trivial at
$\delta^j_\sigma - \varepsilon$ under $d_j$, again contradicting the
definition of $\delta^j_\sigma$.
Thus there exists $(v_\sigma,w_\sigma)\in V\times V$ such that
\[
d_i(v_\sigma,w_\sigma) \le \beta^j_\sigma
\quad\text{and}\quad
d_j(v_\sigma,w_\sigma) \ge \delta^j_\sigma,
\]
and consequently
\[
\delta^j_\sigma - \beta^j_\sigma
\;\le\;
d_j(v_\sigma,w_\sigma) - d_i(v_\sigma,w_\sigma)
\;\le\;
\lVert d_j - d_i \rVert_\infty.
\]

(iii) Finally, the number of birth edges in $\mathcal{B}_j$
is bounded by $|E|-(|V|-1)$.
Indeed, each birth edge corresponds to an edge of $G$ that does not
belong to the minimum spanning tree, and a spanning tree has $|V|-1$
edges; see \cref{prop1}.
Therefore the total number of terms in
\cref{def:tpd} is at most $|E|-|V|+1$.
Since each term in the sums defining $\mathrm{Diff}_p$ is bounded
by $\lVert d_j - d_i \rVert_\infty$, taking the $\ell^p$–norm gives
\[
\mathrm{Diff}_p\!\left(\textbf{bcd}_1(G,d_j),
                      \textbf{bcd}_1(G,d_i)\right)
\;\le\;
(|E|-|V|+1)^{1/p} \,\lVert d_j - d_i \rVert_\infty,
\]
which is the desired bound.
\end{proof}

The constant $\gamma$ in \cref{Thm:tpd-stability} depends only on the combinatorial structure of $G$ through $|V|$ and $|E|$.
The stability bound expresses that the total persistence difference cannot grow faster than linearly with respect to the uniform deviation
between the two distances.
\section{Experiments and Applications}
\label{sec:Applications}
In this section we study how the total persistence difference defined
in \cref{def:tpd} behaves on a real-world time-varying network.
Our primary goal is to compare different persistence-based distances
on the \emph{same} temporal graph and to examine how they relate to
classical graph statistics, rather than to perform a detailed
application-specific analysis of the underlying data.
We use the SNAP EU Research Institution E-Mail dataset as a running
example.\footnote{\url{https://snap.stanford.edu/data/email-Eu-core-temporal.html}}
This dataset records e-mail exchanges between members of a European
research institution over a period of several years.
Following the convention in earlier work~\cite{hajij2018visual}, we
restrict attention to the first $525$ days, after which the published
record becomes extremely sparse and contains long intervals with no
activity.
The same dataset and time range are used in~\cite{hajij2018visual},
so adopting this setting allows for a direct comparison
between our method and their PH-based construction.

\subsection{Daily graph construction and distances}
\label{subsec:daily_graph_construction}

For each day $t\in\{0,\dots,524\}$ we form a simple undirected graph
$G_t=(V_t,E_t,W_t)$ as follows.
The vertex set $V_t = V$ consists of all e-mail accounts appearing in the dataset at $t$.
An undirected edge $(u,v)$ belongs to $E_t$ if at least one e-mail
between $u$ and $v$ was recorded with timestamp in the interval
$[t,t+1)$.
Let
\[
  c_t(u,v)
  \;=\;
  \#\{\text{e-mails between $u$ and $v$ on day $t$}\}
\]
denote the corresponding count.
To interpret more frequent communication as a shorter distance, we
assign the edge weight
\[
  W_t(u,v)
  \;=\;
  \frac{1}{c_t(u,v)},
  \qquad (u,v)\in E_t
\]
for $c_t (u,v) \neq 0$. Note that the case $c_t(u,v) = 0$ does not happen for the edge $(u,v)$ because of the definition.

On each daily graph $G_t$ we consider two path-representable cost-dominated distances $d_{\text{weight}}$ and $d_{\text{edge}}$ defined in \cref{def: distance_weight,def: distance_edge}.
We have $d_{\text{weight}}\le d_{\text{edge}}$ on
each $G_t$, so \cref{Thm: main theorem_general} can be applied.
For each day $t$ and each distance $d\in\{d_{\text{weight}},
d_{\text{edge}}\}$ we compute the $1$-dimensional Vietoris--Rips
persistence barcode $\textbf{bcd}_1(G_t,d)$.
We then evaluate the $\ell^1$-based total persistence difference
\[
  \mathrm{Diff}_1\bigl(\textbf{bcd}_1(G_t,d_{\text{edge}}),
  \textbf{bcd}_1(G_t,d_{\text{weight}})\bigr),
\]
which we abbreviate as \emph{TPD}.

\begin{remark}
  By definition, $\mathrm{Diff}_p$ is an $\ell^p$-norm of the vector of persistence differences associated with all birth edges.
  For a fixed vector, $\ell^p$-norms are nonincreasing in $p$ and become more sensitive to large coordinates as $p$ increases; in   other words, $\mathrm{Diff}_p$ increasingly emphasizes intervals whose change in persistence is large.
  In this application we prefer to weight all units of persistence change linearly and not to over-emphasize a few very large differences, so we work with $\mathrm{Diff}_1$.
\end{remark}

In addition to TPD, for each daily graph $G_t$ we compute several standard graph statistics.
These statistics are not competitors to TPD; instead we use them as simple, interpretable references in order to understand what kinds of
structural changes are being detected by TPD and by the previous PH-based quantity discussed below.
For graphs without edges or with very few vertices we adopt conventions to avoid division by zero.

\begin{itemize}
  \item \textbf{Edge weight standard deviation and variance.}
        We take the standard deviation and variance of the daily edge
        weights $\{W_t(u,v)\}_{(u,v)\in E_t}$.
        When all active edges carry similar weights, these quantities
        are small, whereas highly heterogeneous communication
        patterns lead to larger values.
        If $E_t$ is empty, we define both the standard deviation and
        variance to be $0$.

  \item \textbf{Edge density.}
        We define the edge density of $G_t$ by
        \[
          \rho_t =
          \begin{cases}
            \displaystyle
            \frac{2|E_t|}{|V_t|(|V_t|-1)}, & |V_t|\ge 2,\\[0.7em]
            0, & |V_t|\le 1.
          \end{cases}
        \]
        Thus $\rho_t$ is $0$ for empty or single-vertex graphs and
        increases as edges are added, measuring how densely the daily
        graph is connected.

  \item \textbf{Average degree.}
        We define the average degree of $G_t$ by
        \[
          \bar{k}_t =
          \begin{cases}
            \displaystyle
            \frac{2|E_t|}{|V_t|}, & |V_t|\ge 1,\\[0.7em]
            0, & |V_t|=0.
          \end{cases}
        \]
        This quantity describes the average number of daily contacts
        per account.

  \item \textbf{Average clustering coefficient.}
        For each vertex $v\in V_t$ we set
        \[
          C_t(v) =
          \begin{cases}
            \displaystyle
            \frac{2T_t(v)}{k_t(v)(k_t(v)-1)}, & k_t(v)\ge 2,\\[0.7em]
            0, & k_t(v)<2,
          \end{cases}
        \]
        where $T_t(v)$ is the number of triangles containing $v$ and
        $k_t(v)$ is its degree in $G_t$.
        The average clustering coefficient is then
        \(
          \frac{1}{|V_t|}\sum_{v\in V_t} C_t(v),
        \)
        which summarizes how prevalent tightly knit local
        neighborhoods are in the daily graph.
\end{itemize}

\medskip

To compare with an existing PH-based approach on the same
data, we also reproduce the method of Hajij et al.~\cite{hajij2018visual},
which was originally developed for this dataset.
Their work is the first to propose a systematic pipeline that uses persistent homology to quantify structural changes in time-varying graphs and to visualize those changes over time.
Since our goal is to understand how the proposed total persistence difference behaves compared to such measures computed by the distance between $1$-dimensional persistence barcodes, their method provides a well-developed baseline.
Hajij et al.\ compare bottleneck distance ($p=\infty$) and
$p$-Wasserstein distance on this dataset and argue that the
$p=2$ Wasserstein distance yields a smoother and more interpretable
signal, because it aggregates perturbations across all features.%
\footnote{See Section~4.2.1 and Figure~10 of~\cite{hajij2018visual},
where they contrast noisy timelines produced by bottleneck distance
with smoother timelines obtained from $2$-Wasserstein distance on
$0$- and $1$-dimensional barcodes.}
Following their setting, in our reproduction we focus on the $2$-Wasserstein distance applied to $1$-dimensional persistence barcodes built from shortest-path distances.

We work with the same temporal edges, whose timestamps are given in seconds.
Let $t_{\min}$ and $t_{\max}$ denote the minimum and maximum
days, obtained by rescaling time by a factor $1/86400$, and let
$T = \lfloor t_{\max} \rfloor$.
For each integer day $t \in \{0,\dots,T\}$ we construct an
overlapping time window
\[
  I_t \;=\; [\,\max\{0,\,t - 0.45\},\, t + 1.45\,),
\]
measured in days.
Thus, except for the truncation at $0$ when $t=0$, each window has length $1.9$ days and is centered at $t+0.5$.
Two consecutive windows $I_t$ and $I_{t+1}$ therefore overlap on a subinterval of length $0.9$ days, so that the graph for day $t$ shares $0.45$ days of data with the graph for neighboring days.
From each window $I_t$ we build a simple weighted graph $H_t$ on the same vertex set $V$.
For any two distinct vertices $u$ and $v$ we count the number of e-mails with timestamp in $I_t$,
\[
  c_t^{\mathrm{win}}(u,v)
  \;=\;
  \#\{\text{e-mails between $u$ and $v$ whose time lies in $I_t$}\},
\]
and insert an undirected edge $(u,v)$ in $H_t$ whenever
$c_t^{\mathrm{win}}(u,v) > 0$.
The weight of this edge is taken to be the corresponding count,
\[
  W_t^{\mathrm{win}}(u,v) \;=\; c_t^{\mathrm{win}}(u,v),
  \qquad (u,v)\in E(H_t).
\]
Following~\cite{hajij2018visual}, we equip each $H_t$ with the
weighted shortest-path distance induced by these communication counts
and compute the $1$-dimensional Vietoris--Rips persistence barcode $\textbf{bcd}_1(H_t)$.
Structural change is then quantified by the $2$-Wasserstein distance between consecutive barcodes,
\[
  W_2\bigl(\textbf{bcd}_1(H_t),\,\textbf{bcd}_1(H_{t+1})\bigr),
\]
yielding a scalar time series indexed by $t$.
For brevity we refer to this baseline quantity simply as the
\emph{Wasserstein measure}.
Running this method on the same $525$-day interval produces a
Wasserstein time series that can be compared directly with the time
series of total persistence differences.

Notice the conceptual difference between the two approaches.
The Wasserstein measure compares persistence barcodes obtained from
\emph{different} graphs $H_t$ and $H_{t+1}$ built from overlapping
time windows, thereby encoding structural change across time.
In contrast, TPD compares persistence barcodes built from two
distances on the \emph{same} daily graph $G_t$, so it captures how
the choice of distance function affects $1$-dimensional persistence
\emph{without} the temporal overlap.
This lack of overlap is a key advantage of our framework.

\subsection{Correlation with graph statistics}

The first set of experiments investigates how sensitively TPD and
Wasserstein-based distances respond to changes in basic graph
statistics.
For each day $t$ we have:
\begin{enumerate}
  \item the total persistence difference
        $\mathrm{Diff}_1(\textbf{bcd}_1(G_t,d_{\text{edge}}),
        \textbf{bcd}_1(G_t,d_{\text{weight}}))$;
  \item the Wasserstein measure
        $W_2(\textbf{bcd}_1(H_t),\textbf{bcd}_1(H_{t+1}))$
        between overlapping-window graphs; and
  \item the graph statistics described above.
\end{enumerate}
We compute Pearson correlation coefficients between each of the
topological quantities and each graph statistic.
The Pearson coefficient $r\in[-1,1]$ measures linear dependence,
with $|r|$ close to $1$ indicating strong correlation~\cite{pearson1895vii}; we also report the associated
$p$-values.
\cref{fig:correlation_plots} shows scatter plots of TPD and Wasserstein measure against average degree and clustering coefficient, each point corresponding to $G_t$ and $H_t$ respectively.
\cref{tab:comparison_corr} summarizes the correlation values
for both TPD and Wasserstein measure.
Overall, TPD exhibits substantially stronger correlation with average
degree and clustering coefficient than the Wasserstein measure based
on overlapping windows.
The purpose of introducing these graph statistics is not to propose them as alternatives to PH-based descriptors.
They are standard, coarse summaries of connectivity that are widely used in network analysis, and they provide an interpretable sanity check: any PH-based quantity whose behavior is completely decoupled from such basic descriptors would be difficult to justify in applications.
In our setting we therefore use them only as a neutral reference scale against which we can compare two PH-based constructions on the same data, namely the Wasserstein measure and TPD.
\begin{figure}[h]
    \centering
    \includegraphics{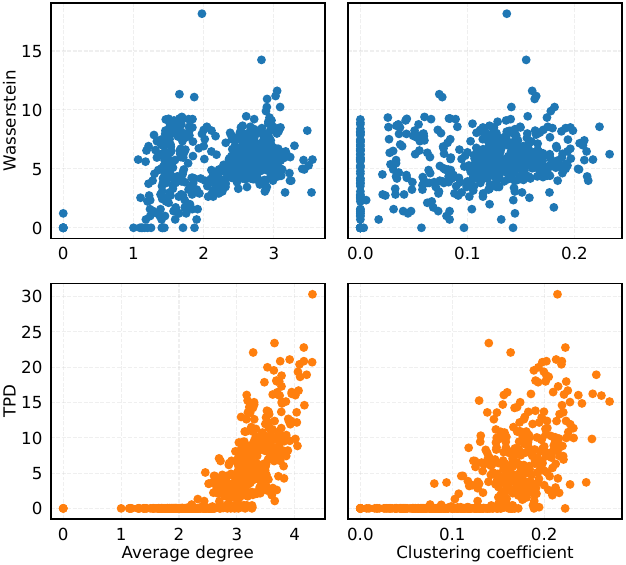}
    \caption{Scatter plots between topological distances and graph statistics.
    Top row: Wasserstein measures between overlapping-window graphs versus average degree (left) and clustering coefficient (right).
    Bottom row: total persistence difference between
    $d_{\text{weight}}$ and $d_{\text{edge}}$ on the same daily graph versus the same statistics.
    Each dot represents a single day.}
    \label{fig:correlation_plots}
\end{figure}
\begin{table}[h]
\centering
\caption{Pearson correlation between graph statistics and topological distances.}
\label{tab:comparison_corr}
\begin{tabular}{lcccc}
\toprule
\multirow{2}{*}{\textbf{Statistic}} & \multicolumn{2}{c}{\textbf{Previous Result}} & \multicolumn{2}{c}{\textbf{Total Persistence Difference}} \\
\cmidrule(lr){2-3} \cmidrule(lr){4-5}
& Correlation ($r$) & $p$-value & Correlation ($r$) & $p$-value \\
\midrule
Edge weight std.        & 0.141 & $1.23 \times 10^{-3}$ & 0.352 & $1.02 \times 10^{-16}$ \\
Edge weight variance    & 0.053 & $2.22 \times 10^{-1}$ & 0.413 & $5.09 \times 10^{-23}$ \\
Density                 & $-0.181$ & $3.01 \times 10^{-5}$ & $-0.450$ & $1.77 \times 10^{-27}$ \\
Average degree          & 0.362 & $1.25 \times 10^{-17}$ & \textbf{0.743} & $4.39 \times 10^{-93}$ \\
Clustering coefficient  & 0.394 & $7.38 \times 10^{-21}$ & \textbf{0.611} & $4.03 \times 10^{-55}$ \\
\bottomrule
\end{tabular}
\end{table}

The graph statistics are both computed on $G_t$ and $H_t$, to compare with TPD and Wasserstein measure fairly.
The fact that TPD aligns better with the basic graph statistics than the Wasserstein measure suggests that the difference between two distances on a fixed graph is able to track meaningful structural changes, rather than behaving as an overly abstract quantity.
At the same time TPD depends on $1$-dimensional homology, so it is sensitive to mesoscale cycle structure that is invisible to degree, density, or clustering alone.

\subsection{Weekday effects and inactivity periods}

Daily communication naturally varies across weekdays and weekends.
To visualize this behavior, we min--max normalize both the Wasserstein measure and TPD separately and group values by weekday, as the two quantities live on different numerical scales.
Precisely, for each method we take the raw sequence $(z_t)_{t=0}^{524}$ and calculate
\[
  z_t^{\mathrm{norm}}
  =
  \frac{z_t - \min_{0\le s\le 524} z_s}
       {\max_{0\le s\le 524} z_s - \min_{0\le s\le 524} z_s},
\]
so that each normalized series takes values in $[0,1]$.
\cref{fig:weekday_plot} shows box plots of the normalized values grouped by weekday.

\begin{figure}[h]
    \centering
    \includegraphics{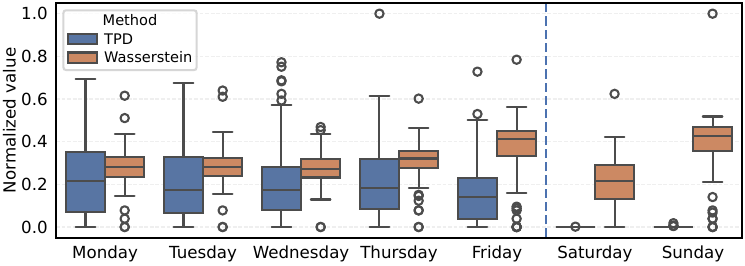}
    \caption{Box plots of normalized Wasserstein measure and normalized TPD grouped by weekday.
    The vertical dashed line separates Friday from Saturday and Sunday.
    TPD clearly decays on weekends, while the Wasserstein measure remains high.}
    \label{fig:weekday_plot}
\end{figure}

The normalized Wasserstein measure does not decay on weekends, even though many weekend graphs are structurally very simple or nearly empty.
This is consistent with the construction: each window $I_t$ contains $0.45$ days of data from both the preceding and following days, so a weekend window still includes a substantial amount of weekday traffic.
As a consequence, the Wasserstein measure compares persistence barcodes built from graphs that mix weekend and weekday edges, and it does not fully reflect how the data at a specific point of time behave.
In contrast, TPD is computed on single-day graphs $G_t$ without any
temporal overlap.
For many weekend days the corresponding daily graphs are so sparse
that $d_{\text{edge}}$ and $d_{\text{weight}}$ induce identical
$1$-dimensional persistence barcodes, and the total persistence
difference becomes zero.
A value of $\mathrm{Diff}_1 = 0$ does not mean that there were
no e-mails on that day; rather, it means that, given the simple
connectivity pattern, changing the distance from $d_{\text{weight}}$
to $d_{\text{edge}}$ does not create or destroy any $1$-dimensional
features or change their death values.
In this sense, zeros of TPD identify days on which the choice of
distance has no visible effect on $1$-dimensional topology.
We do not claim that being zero is always the desirable behavior
for every application.
Here the role of TPD is to measure the \emph{additional} effect of
changing the distance on a fixed graph.
From that perspective, it is natural and even desirable that TPD
vanishes on days when the graph has very simple cycle structure and
the two distances behave similarly, and that it takes larger values
on days when the graph supports many alternative paths and the two distances disagree.
In practice we envision TPD as an add-on feature: it should be used
alongside other PH-based or graph-based indicators, not only as a
standalone summary of activity, and its main contribution is to
highlight time periods in which the underlying distance choice matters.

\subsection{Graphs with low and high total persistence difference}

To gain structural intuition, we examine two representative daily
graphs: one with the smallest nonzero TPD among the weekdays (Day 191) and one with the largest TPD (Day 486).
Their TPD values are $0.05$ and $30.3131$, respectively.
In selecting the low-TPD example we excluded trivial and erratic cases: we restricted attention to weekdays and set a threshold $\mathrm{TPD} \ge 10^{-3}$.
Among the remaining days, Day 191 attains the smallest TPD.\footnote{Among all days with $0 <$ TPD $< 10^{-3}$ the
largest value is on the order of $2\times 10^{-16}$, which is
comparable to floating-point roundoff and is therefore treated as
numerical zero.}
\begin{figure}[h]
    \centering
    \includegraphics{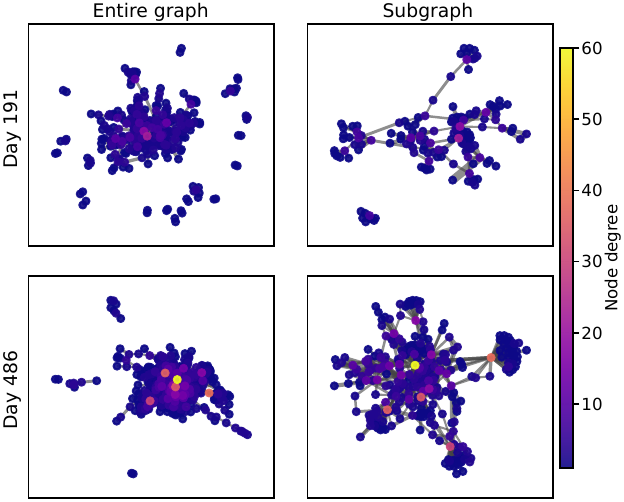}
    \caption{Daily graphs with low (Day 191) and high (Day 486)
      total persistence differences and their largest connected
      components.
      The low-difference graph is sparse with few cycles, whereas the
      high-difference graph is denser and contains many overlapping
      routes between vertices.}
    \label{fig:lowhigh_graphs}
\end{figure}

In the low-difference case the daily graph is visually sparse: many
vertices have small degree and there are relatively few edges
connecting different parts of the network.
Under such conditions, both $d_{\text{edge}}$ and
$d_{\text{weight}}$ see a similar family of paths as candidates,
so the associated $1$-dimensional persistence barcodes are similar
and TPD is small.

In the high-difference case the graph is considerably denser: many
vertices have moderate to high degree and a larger number of edges
with various weights.
Here it is easier for $d_{\text{edge}}$ and $d_{\text{weight}}$ to disagree on
which routes are ``shortest'': edges with large weight are relatively
unattractive under $d_{\text{weight}}$, but they still contribute unit
length under $d_{\text{edge}}$.
As a result the two distances can select different families of short
connections, which leads to noticeable changes in the death values in $1$-dimensional barcodes and hence to a large total persistence difference.
Thus the visual examples are consistent with the idea that TPD becomes large when the graph supports many alternative path configurations, and remains small when the connectivity is simple.

Finally, we compare the $1$-dimensional persistence barcodes for the
two representative days under both distances.
\cref{fig:barcode_comparison} shows barcodes based on
$d_{\text{edge}}$ (top row) and $d_{\text{weight}}$ (bottom row) for the
low- and high-difference days.
For the low-difference day, both distances yield relatively few
intervals, most of which are short.
In the high-difference case there are many more intervals, and the
difference in death values between the two distances is pronounced.
These examples illustrate how the total persistence difference
summarizes, in a single quantity, the cumulative change in
persistence values induced by switching the distance function on the
same underlying graph.

\begin{figure}[!htbp]
    \centering
    \includegraphics{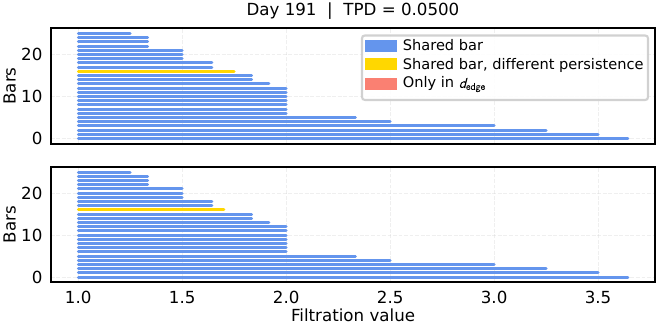}
    \includegraphics{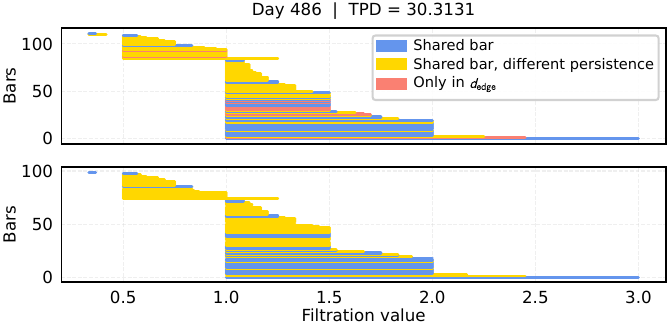}
    \caption{One-dimensional persistence barcodes on the low- and
    high-TPD days under the two distance functions.
    Top: Day~191 (TPD $=0.05$); bottom: Day~486 (TPD $=30.3131$).
    In each panel, the upper barcode is built from $d_{\text{edge}}$
    and the lower barcode from $d_{\text{weight}}$.
    Blue intervals appear in both barcodes with identical birth and death
    values $(\beta,\delta)$; yellow intervals appear in both but with
    different persistence; and red intervals appear only in the
    $d_{\text{edge}}$-based barcode.
    The numbers of intervals are $26$, $26$, $112$, and $100$ from top to
    bottom.}
    \label{fig:barcode_comparison}
\end{figure}

Taken together, these experiments suggest that the total persistence
difference is a useful complement to existing persistence-based
distances on time-varying graphs.
When the daily graphs are structurally simple, TPD is often close to
zero and existing methods such as persistence landscapes can capture topological features.
As the graphs become more complex and support many alternative routes, TPD increases and correlates well with global connectivity statistics such as average degree and clustering.
From a practical standpoint, TPD can therefore be utilized as an
\emph{add-on} quantity: it can be used alongside existing PH-based
measures to highlight regimes where the choice of distance function matters, while it becomes less sensitive whenever the graph is simple enough.
\section{Conclusion}\label{sec:Conclusion}

This paper investigates how different path-based distances on a fixed
graph affect one-dimensional persistent homology.
We focused on two natural choices on a weighted graph:
the distance $d_{\text{weight}}$ defined by the minimum sum of edge
weights along a path, and the distance $d_{\text{edge}}$ that
prioritizes paths with the smallest number of edges over
their weights.
To analyze such distances in a common framework we introduced the
notion of a path-representable distance and, within this class,
identified a cost-dominated condition in which the Vietoris--Rips
filtration admits a particular structure.

Our main theoretical result shows that, whenever
$d_i \leq d_j$ are cost-dominated path-representable distances on the
same connected graph, there exists an injective map between their
one-dimensional persistence barcodes that preserves birth values and
maps each death value to an equal or larger one.
This provides a precise relationship between the homological features
detected by different distances and clarifies how changes in the path
selection can affect the $1$-dimensional persistence.

Motivated by this structure, we introduced the \emph{total persistence
difference}, which quantifies the cumulative change in death values
between the two barcodes respecting the correspondence induced by birth edges.
We established a Lipschitz stability bound for this quantity in
terms of the $L^\infty$-difference between the underlying distances.
Numerical experiments on a temporal e-mail network suggest that the
total persistence difference can track structural changes in the graph
and often exhibits strong correlation with basic graph statistics, such as average degree and clustering coefficient.
At the same time, the experiments indicate that when the daily graphs
are structurally simple the total persistence difference is frequently
very small or even zero.
From a practical point of view, this suggests that the total
persistence difference can be used as a complementary quantity.

Several directions remain open.
The present work treats path-representable distances that follow a
single path and satisfy the consistency condition; more general
graph-based distances, or distances defined by combining several
candidate paths, are not covered directly by our framework.
The main theorem and the subsequent stability results are proved for the VR filtration on undirected graphs; extending these ideas to other filtrations or to directed graphs would be interesting topics for future research.
Finally, the injection structure and the total persistence difference could be explored in other application domains where multiple distance choices on the same graph naturally arise and may carry domain-specific meaning.

\vskip .2in
\noindent
{\bf{Acknowledgements:} }
This work was supported by National Research Foundation of Korea under grant number 2018R1D1A1A02086097 and grant number 2021R1A2C3009648. This work was also supported partially by the KIAS Transdisciplinary program grant.

\bibliographystyle{plain}
\bibliography{reference_main}

@article{Cohen,
	author = {Cohen-Steiner, David and Edelsbrunner, Herbert and Harer, John},
	journal = {Discrete and Computational Geometry},
	title = {Stability of Persistence Diagrams},
	volume = {37103--120},
	year = {2007}}

@article{Carlsson,
	author = {Gunnar Carlsson},
	journal = {Bulletin of The American Mathematical Society},
	pages = {255--308},
	title = {Topology and data},
	volume = {46},
	year = {2009}}

@book{EHarer,
	author = {Herbert Edelsbrunner and John L. Harer},
	publisher = {American Mathematical Society},
	title = {Computational Topology: An Introduction},
	year = {2009}}

@article{ZC2005,
	author = {Zomorodian, Afra and Carlsson, Gunnar},
	doi = {10.1007/s00454-004-1146-y},
	journal = {Discrete and Computational Geometry},
	month = {02},
	pages = {249-274},
	title = {Computing Persistent Homology},
	volume = {33},
	year = {2005}}

@article{Aktas,
	author = {Aktas, M.E. and  Akbas, E. and Fatmaoui, A.E.},
	journal = {Applied Network Science},
	month = {},
	pages = {},
	title = {Persistence homology of networks: methods and applications},
	volume = {4}, number = {61},
	year = {2019},
	}

@inproceedings{cohen2006vines,
  title={Vines and vineyards by updating persistence in linear time},
  author={Cohen-Steiner, David and Edelsbrunner, Herbert and Morozov, Dmitriy},
  booktitle={Proceedings of the twenty-second annual symposium on Computational geometry},
  pages={119--126},
  year={2006}
}

@article{tran2023topological,
author = {Mai Lan Tran and Changbom Park and Jae-Hun Jung},
title = {Topological data analysis of Korean music in Jeongganbo: a cycle structure},
journal = {Journal of Mathematics and Music},
volume = {17},
number = {3},
pages = {403--432},
year = {2023},
publisher = {Taylor \& Francis},
doi = {10.1080/17459737.2022.2164626},
URL = { 
        https://doi.org/10.1080/17459737.2022.2164626
},
eprint = { 
        https://doi.org/10.1080/17459737.2022.2164626
}
}

@article{sizemore_netneurtda,
    author = {Sizemore, Ann E. and Phillips-Cremins, Jennifer E. and Ghrist, Robert and Bassett, Danielle S.},
    title = "{The importance of the whole: Topological data analysis for the network neuroscientist}",
    journal = {Network Neuroscience},
    volume = {3},
    number = {3},
    pages = {656-673},
    year = {2019},
    month = {07},
    issn = {2472-1751},
    doi = {10.1162/netn_a_00073},
    url = {https://doi.org/10.1162/netn\_a\_00073},
    eprint = {https://direct.mit.edu/netn/article-pdf/3/3/656/1092395/netn\_a\_00073.pdf},
}

@article{stolz_phfunc,
    author = {Stolz, Bernadette J. and Harrington, Heather A. and Porter, Mason A.},
    title = "{Persistent homology of time-dependent functional networks constructed from coupled time series}",
    journal = {Chaos: An Interdisciplinary Journal of Nonlinear Science},
    volume = {27},
    number = {4},
    pages = {047410},
    year = {2017},
    month = {04},
    issn = {1054-1500},
    doi = {10.1063/1.4978997},
    url = {https://doi.org/10.1063/1.4978997},
    eprint = {https://pubs.aip.org/aip/cha/article-pdf/doi/10.1063/1.4978997/13299205/047410\_1\_online.pdf},
}

@ARTICLE{leebraintda,
  author={Lee, Hyekyoung and Kang, Hyejin and Chung, Moo K. and Kim, Bung-Nyun and Lee, Dong Soo},
  journal={IEEE Transactions on Medical Imaging}, 
  title={Persistent Brain Network Homology From the Perspective of Dendrogram}, 
  year={2012},
  volume={31},
  number={12},
  pages={2267-2277},
  doi={10.1109/TMI.2012.2219590}}

@InProceedings{gideafinance,
author="Gidea, Marian",
editor="Shmueli, Erez
and Barzel, Baruch
and Puzis, Rami",
title="Topological Data Analysis of Critical Transitions in Financial Networks",
booktitle="3rd International Winter School and Conference on Network Science ",
year="2017",
publisher="Springer International Publishing",
address="Cham",
pages="47--59",
isbn="978-3-319-55471-6"
}

@Article{gholizadehtext,
AUTHOR = {Gholizadeh, Shafie and Seyeditabari, Armin and Zadrozny, Wlodek},
TITLE = {Topological Signature of 19th Century Novelists: Persistent Homology in Text Mining},
JOURNAL = {Big Data and Cognitive Computing},
VOLUME = {2},
YEAR = {2018},
NUMBER = {4},
ARTICLE-NUMBER = {33},
URL = {https://www.mdpi.com/2504-2289/2/4/33},
ISSN = {2504-2289},
DOI = {10.3390/bdcc2040033}
}

@unknown{gholizadehtext2,
author = {Gholizadeh, Shafie and Savle, Ketki and Seyeditabari, Armin and Zadrozny, Wlodek},
year = {2020},
month = {03},
pages = {},
title = {Topological Data Analysis in Text Classification: Extracting Features with Additive Information},
doi = {10.48550/arXiv.2003.13138}
}

@article{mehmetmakam,
author = {Mehmet Emin Aktas and Esra Akbas and Jason Papayik and Yunus Kovankaya},
title = {Classification of Turkish makam music: a topological approach},
journal = {Journal of Mathematics and Music},
volume = {13},
number = {2},
pages = {135--149},
year = {2019},
publisher = {Taylor \& Francis},
doi = {10.1080/17459737.2019.1622810},

URL = { 
        https://doi.org/10.1080/17459737.2019.1622810
},
eprint = { 
        https://doi.org/10.1080/17459737.2019.1622810
}
}

@article{MAJUMDAR2020113868,
title = {Clustering and classification of time series using topological data analysis with applications to finance},
journal = {Expert Systems with Applications},
volume = {162},
pages = {113868},
year = {2020},
issn = {0957-4174},
doi = {https://doi.org/10.1016/j.eswa.2020.113868},
url = {https://www.sciencedirect.com/science/article/pii/S095741742030676X},
author = {Sourav Majumdar and Arnab Kumar Laha}
}

@article{GIDEA2018820,
title = {Topological data analysis of financial time series: Landscapes of crashes},
journal = {Physica A: Statistical Mechanics and its Applications},
volume = {491},
pages = {820-834},
year = {2018},
issn = {0378-4371},
doi = {https://doi.org/10.1016/j.physa.2017.09.028},
url = {https://www.sciencedirect.com/science/article/pii/S0378437117309202},
author = {Marian Gidea and Yuri Katz},
}

@article{ahujashortpath,
author = {Ahuja, Ravindra K. and Mehlhorn, Kurt and Orlin, James and Tarjan, Robert E.},
title = {Faster algorithms for the shortest path problem},
year = {1990},
issue_date = {April 1990},
publisher = {Association for Computing Machinery},
address = {New York, NY, USA},
volume = {37},
number = {2},
issn = {0004-5411},
url = {https://doi.org/10.1145/77600.77615},
doi = {10.1145/77600.77615},
journal = {J. ACM},
month = apr,
pages = {213–223},
numpages = {11}
}

@article{cherkasskypriority,
author = {Cherkassky, Boris V. and Goldberg, Andrew V. and Silverstein, Craig},
title = {Buckets, Heaps, Lists, and Monotone Priority Queues},
journal = {SIAM Journal on Computing},
volume = {28},
number = {4},
pages = {1326-1346},
year = {1999},
doi = {10.1137/S0097539796313490},
URL = { 
        https://doi.org/10.1137/S0097539796313490
},
eprint = { 
        https://doi.org/10.1137/S0097539796313490
}
}

@article{FREDMAN1994533,
title = {Trans-dichotomous algorithms for minimum spanning trees and shortest paths},
journal = {Journal of Computer and System Sciences},
volume = {48},
number = {3},
pages = {533-551},
year = {1994},
issn = {0022-0000},
doi = {https://doi.org/10.1016/S0022-0000(05)80064-9},
url = {https://www.sciencedirect.com/science/article/pii/S0022000005800649},
author = {Michael L. Fredman and Dan E. Willard},
}

@article{Cherkassky1996,
author={Cherkassky, Boris V.
and Goldberg, Andrew V.
and Radzik, Tomasz},
title={Shortest paths algorithms: Theory and experimental evaluation},
journal={Mathematical Programming},
year={1996},
month={May},
day={01},
volume={73},
number={2},
pages={129-174},
issn={1436-4646},
doi={10.1007/BF02592101},
url={https://doi.org/10.1007/BF02592101}
}

@InProceedings{goldbergsimple,
author="Goldberg, Andrew V.",
editor="auf der Heide, Friedhelm Meyer",
title="A Simple Shortest Path Algorithm with Linear Average Time",
booktitle="Algorithms --- ESA 2001",
year="2001",
publisher="Springer Berlin Heidelberg",
address="Berlin, Heidelberg",
pages="230--241",
}

@inproceedings{zhushortest,
author = {Zhu, Andy Diwen and Ma, Hui and Xiao, Xiaokui and Luo, Siqiang and Tang, Youze and Zhou, Shuigeng},
title = {Shortest path and distance queries on road networks: towards bridging theory and practice},
year = {2013},
isbn = {9781450320375},
publisher = {Association for Computing Machinery},
address = {New York, NY, USA},
url = {https://doi.org/10.1145/2463676.2465277},
doi = {10.1145/2463676.2465277},
booktitle = {Proceedings of the 2013 ACM SIGMOD International Conference on Management of Data},
pages = {857–868},
numpages = {12},
location = {New York, New York, USA},
series = {SIGMOD '13}
}

@INPROCEEDINGS{RIZI2018IEEE,
  author={Rizi, Fatemeh Salehi and Schloetterer, Joerg and Granitzer, Michael},
  booktitle={2018 IEEE/ACM International Conference on Advances in Social Networks Analysis and Mining (ASONAM)}, 
  title={Shortest Path Distance Approximation Using Deep Learning Techniques}, 
  year={2018},
  volume={},
  number={},
  pages={1007-1014},
  doi={10.1109/ASONAM.2018.8508763}}

@inproceedings{POTAMIAS2009ACM,
author = {Potamias, Michalis and Bonchi, Francesco and Castillo, Carlos and Gionis, Aristides},
title = {Fast shortest path distance estimation in large networks},
year = {2009},
isbn = {9781605585123},
publisher = {Association for Computing Machinery},
address = {New York, NY, USA},
url = {https://doi.org/10.1145/1645953.1646063},
doi = {10.1145/1645953.1646063},
booktitle = {Proceedings of the 18th ACM Conference on Information and Knowledge Management},
pages = {867–876},
numpages = {10},
location = {Hong Kong, China},
series = {CIKM '09}
}

@inproceedings{NEURIPS2020_2f73168b,
 author = {Li, Pan and Wang, Yanbang and Wang, Hongwei and Leskovec, Jure},
 booktitle = {Advances in Neural Information Processing Systems},
 editor = {H. Larochelle and M. Ranzato and R. Hadsell and M.F. Balcan and H. Lin},
 pages = {4465--4478},
 publisher = {Curran Associates, Inc.},
 title = {Distance Encoding: Design Provably More Powerful Neural Networks for Graph Representation Learning},
 url = {https://proceedings.neurips.cc/paper_files/paper/2020/file/2f73168bf3656f697507752ec592c437-Paper.pdf},
 volume = {33},
 year = {2020}
}

@InProceedings{pmlr-v198-abboud22a,
  title = {Shortest Path Networks for Graph Property Prediction},
  author = {Abboud, Ralph and Dimitrov, Radoslav and Ceylan, Ismail Ilkan},
  booktitle = {Proceedings of the First Learning on Graphs Conference},
  pages = {5:1--5:25},
  year = {2022},
  editor = {Rieck, Bastian and Pascanu, Razvan},
  volume = {198},
  series = {Proceedings of Machine Learning Research},
  month = {09--12 Dec},
  publisher = {PMLR},
  url = 	 {https://proceedings.mlr.press/v198/abboud22a.html},
}

@INPROCEEDINGS{JI2021IEEE,
  author={Ji, Houye and Yang, Cheng and Shi, Chuan and Li, Pan},
  booktitle={2021 IEEE International Conference on Data Mining (ICDM)}, 
  title={Heterogeneous Graph Neural Network with Distance Encoding}, 
  year={2021},
  volume={},
  number={},
  pages={1138-1143},
  doi={10.1109/ICDM51629.2021.00135}}

@InProceedings{pmlr-v202-michel23a,
  title = 	 {Path Neural Networks: Expressive and Accurate Graph Neural Networks},
  author =       {Michel, Gaspard and Nikolentzos, Giannis and Lutzeyer, Johannes F. and Vazirgiannis, Michalis},
  booktitle = 	 {Proceedings of the 40th International Conference on Machine Learning},
  pages = 	 {24737--24755},
  year = 	 {2023},
  editor = 	 {Krause, Andreas and Brunskill, Emma and Cho, Kyunghyun and Engelhardt, Barbara and Sabato, Sivan and Scarlett, Jonathan},
  volume = 	 {202},
  series = 	 {Proceedings of Machine Learning Research},
  month = 	 {23--29 Jul},
  publisher =    {PMLR},
  url = 	 {https://proceedings.mlr.press/v202/michel23a.html},
}

@inproceedings{li2021distance,
  title={Distance-enhanced graph neural network for link prediction},
  author={Li, Boning and Xia, Yingce and Xie, Shufang and Wu, Lijun and Qin, Tao},
  booktitle={ICML 2021 Workshop on Computational Biology},
  year={2021}
}

@inproceedings{abboud2022shortest,
  title={Shortest path networks for graph property prediction},
  author={Abboud, Ralph and Dimitrov, Radoslav and Ceylan, Ismail Ilkan},
  booktitle={Learning on Graphs Conference},
  pages={5--1},
  year={2022},
  organization={PMLR}
}

@inproceedings{hajij2018visual,
  title={Visual detection of structural changes in time-varying graphs using persistent homology},
  author={Hajij, Mustafa and Wang, Bei and Scheidegger, Carlos and Rosen, Paul},
  booktitle={2018 ieee pacific visualization symposium (pacificvis)},
  pages={125--134},
  year={2018},
  organization={IEEE}
}

@article{pearson1895vii,
  title={VII. Note on regression and inheritance in the case of two parents},
  author={Pearson, Karl},
  journal={proceedings of the royal society of London},
  volume={58},
  number={347-352},
  pages={240--242},
  year={1895},
  publisher={The Royal Society London}
}

@article{asao2025classification,
  title={Classification of metric fibrations},
  author={Asao, Yasuhiko},
  journal={Algebraic \& Geometric Topology},
  volume={25},
  number={7},
  pages={4257--4285},
  year={2025},
  publisher={Mathematical Sciences Publishers}
}

@incollection{deza2009encyclopedia,
  title={Encyclopedia of distances},
  author={Deza, Michel Marie and Deza, Elena},
  booktitle={Encyclopedia of distances},
  pages={1--583},
  year={2009},
  publisher={Springer}
}

@article{bubenik2014categorification,
  title={Categorification of persistent homology},
  author={Bubenik, Peter and Scott, Jonathan A},
  journal={Discrete \& Computational Geometry},
  volume={51},
  number={3},
  pages={600--627},
  year={2014},
  publisher={Springer}
}

@article{bergomi2017topological,
  title={Topological graph persistence},
  author={Bergomi, Mattia G and Ferri, Massimo and Zuffi, Lorenzo and others},
  journal={Commun Appl Ind Math},
  volume={11},
  number={1},
  pages={72--87},
  year={2017}
}

@article{bergomi2021beyond,
  title={Beyond topological persistence: Starting from networks},
  author={Bergomi, Mattia G and Ferri, Massimo and Vertechi, Pietro and Zuffi, Lorenzo},
  journal={Mathematics},
  volume={9},
  number={23},
  pages={3079},
  year={2021},
  publisher={MDPI}
}

@article{bergomi2019rank,
  title={Rank-based persistence},
  author={Bergomi, Mattia G and Vertechi, Pietro},
  journal={arXiv preprint arXiv:1905.09151},
  year={2019}
}

@article{mccleary2018bottleneck,
  title={Bottleneck stability for generalized persistence diagrams},
  author={McCleary, Alex and Patel, Amit},
  journal={arXiv preprint arXiv:1806.00170},
  year={2018}
}

\FloatBarrier
\appendix
\section{Topological Data Analysis on Graphs}\label{app:tda-graphs}

\providecommand{\wGraph}{\mathbf{wGraph}}
\providecommand{\SemiMet}{\mathbf{SemiMet}}
\providecommand{\FiltSimp}{\mathbf{FiltSimp}}
\providecommand{\PersMod}{\mathbf{PersMod}}
\providecommand{\VR}{\mathrm{VR}}
\providecommand{\R}{\mathbb{R}}

Although our data throughout this paper are graphs, persistent
homology is not applied directly to the graph structure.
Instead, one has to specify how a graph is turned into a suitable
object -- a filtered simplicial complex or a persistence module -- on which the standard machinery of
topological data analysis acts.
The aim of this appendix is twofold.
First, we briefly survey several existing approaches and theories to develop
persistent homology for graph data.
Second, we formalize the specific graph-to-barcode pipeline used in
this paper, emphasizing its functorial nature and its position within
the classical (semi)metric-based framework of persistent homology.

\subsection{Persistent homology for graph data}
\label{appsec:tda-graph-survey}

Persistent homology was originally developed for filtrations of
topological spaces or abstract simplicial complexes.
A graph $G=(V,E)$ can, of course, be regarded as a $1$-dimensional
simplicial complex, but this viewpoint alone leaves little room for
filtrations beyond trivial edge-addition schemes.
Consequently, applying persistent homology to graphs requires
additional structure:
one must either enrich the graph into a higher-dimensional simplicial
complex (for instance via clique complexes), or adopt a categorical
generalization of persistence that acts directly on graphs and
networks.
Both directions have been explored in the literature.

\subsubsection{Topological constructions on graphs}

A first family of approaches starts from a weighted undirected graph
and associates to it one or more simplicial complexes whose homology
captures non-trivial higher-dimensional holes.
A classical example is the clique complex: given a simple graph
$G$, one declares a finite subset of vertices to be a simplex whenever
it forms a complete subgraph of $G$.
When edge weights are present, these complexes can be equipped with
natural filtrations, such as the weight rank clique filtration, where
edges, and hence cliques, appear as the threshold on edge weights is varied.
Bergomi et al.~\cite{bergomi2017topological} formalize
this idea under the name \emph{topological graph persistence}.
Starting from a weighted graph, they construct several associated
simplicial complexes---including clique, independent-set,
neighborhood, and enclaveless complexes---and study their persistent
homology under various graph-induced filtrations.  In this way the
graph is used as a combinatorial skeleton that generates richer
topological objects; persistent homology is then computed on these
auxiliary complexes, not on the bare graph itself. 

Other works take a similar stance but focus on particular application
domains, such as networks in statistical mechanics or biological
systems, where the graph serves as an intermediate representation and
the actual persistent homology is computed on the graph itself as a $1$-dimensional simplicial complex.
In all these cases, the pipeline has the schematic form
\[
  \mathrm{Graph}
  \;\longrightarrow\;
  \mathrm{Filtered\;simplicial\;complex}
  \;\longrightarrow\;
  \mathrm{Persistence\;module}
  \;\longrightarrow\;
  \mathrm{Barcode}.
\]
The choice of graph-to-complex construction (cliques, neighborhoods,
hyperedges, \emph{etc}.) and of filtration (e.g.\ weight thresholds)
determines which geometric or combinatorial aspects of the network are
highlighted.

\subsubsection{Categorical and network-based generalizations}

A second line of work generalizes persistent homology beyond
topological spaces, using a categorical perspective.
In this setting a persistence object is viewed as a functor
\(
  F \colon (\mathbb{R},\le) \to \mathcal{C}
\)
from the poset of scale parameters to a target category $\mathcal{C}$,
and persistence invariants are constructed in a way that depends only
on functorial data.

Bubenik and Scott~\cite{bubenik2014categorification} recast
classical persistent homology in exactly this language.  They consider
barcodes indexed by $(\mathbb{R},\leq)$ in a suitable category
(e.g.\ vector spaces) and equip the space of such barcodes with an
interleaving distance that extends the usual bottleneck distance on
persistence barcodes.  This categorical viewpoint leads to unified and
general stability theorems for several variants of persistence.
Building on this, Bergomi and Vertechi introduce
\emph{rank-based persistence}~\cite{bergomi2019rank}.
They axiomatize a generalized rank function on objects in a target category $\mathcal{C}$ so that any functor
$F \colon (\mathbb{R},\le) \to \mathcal{C}$ induces a persistence function and an associated persistence barcode.
This framework encompasses both topological persistence (where $\mathcal{C} = \mathrm{Vect}$ and rank is ordinary vector-space dimension) and more combinatorial settings, and it admits natural interleaving and bottleneck metrics. 

In later work, Bergomi, Ferri, Vertechi, and Zuffi develop a
network-centred perspective in which persistence functions are defined
directly on categories of graphs and digraphs, without first passing
to topological spaces~\cite{bergomi2021beyond}.  They introduce
\emph{categorical persistence functions} that measure strong forms of
connectedness---such as clique communities, $k$-vertex and $k$-edge
connectivity, and strong connectivity in digraphs---and show that
these invariants enjoy robustness and computability properties
analogous to classical persistent homology.

The stability of such generalized persistence diagrams is studied by
McCleary and Patel~\cite{mccleary2018bottleneck}, who extend
bottleneck stability to one-dimensional constructible persistence
modules valued in an arbitrary small abelian category.
Their results ensure that, under appropriate finiteness and positivity
assumptions, generalized persistence diagrams carry a natural
bottleneck metric that is stable with respect to interleavings of the
underlying functors.

From the point of view of this work, these categorical frameworks
serve two purposes.
First, they clarify that what ultimately matters is not the ambient
topological space but the functorial structure of the filtration:
once a graph has been transformed into a persistence module, the
usual notions of interleaving, bottleneck distance, and stability
apply.
Second, they provide conceptual justification for viewing our own
graph-based constructions as functors between appropriate categories,
as we now describe.

\subsection{From graphs to semimetric spaces via path-representable distances}
\label{appsec:graph-to-semimetric}

In this work we follow the classical route of representing a graph
as a semimetric space and then applying Vietoris--Rips persistent
homology to that space.
This keeps the analysis within the standard setting of persistent
homology on (semi)metric spaces, while still encoding non-trivial
graph structure via path-representable distances.
For precise definitions of path-representable and cost-dominated distances used in the main text, see \cref{sec:Main theory,sec:Path-representable distance}.

\subsubsection{Categories of graphs and semimetric spaces}

We begin by fixing a simple categorical framework.
Let $\mathbf{wGraph}$ denote the category whose objects are connected,
weighted, undirected simple graphs
\[
  G = (V(G), E(G), W_G),
\]
where $V(G)$ is a finite vertex set, $E(G) \subseteq \{\{u,v\}\subset V(G)\}$ is a set of undirected edges, and
$W_G \colon E(G) \to \mathbb{R}_{>0}$ assigns a positive weight to each
edge.
A morphism $\varphi \colon G \to G'$ in $\mathbf{wGraph}$ is a map
$\varphi \colon V(G) \to V(G')$ such that
\begin{itemize}
  \item if $\{u,v\}\in E(G)$ then $\{\varphi(u),\varphi(v)\}\in E(G')$, and
  \item $W_{G'}(\varphi(u),\varphi(v)) \leq W_G(u,v)$ for every
        $\{u,v\}\in E(G)$.
\end{itemize}
Thus $\varphi$ is a weight-non-increasing graph homomorphism.~\cite{asao2025classification}

Next, let $\mathbf{SemiMet}$ denote the category of finite semimetric
spaces.
An object is a pair $(X,d)$ where $X$ is a finite set and $d$ is a
semimetric on $X$.
A morphism $f \colon (X,d_X) \to (Y,d_Y)$ is a
\emph{non-expansive} map, i.e.
\[
  d_Y\bigl(f(x),f(x')\bigr)
  \;\leq\;
  d_X(x,x')
  \quad\text{for all }x,x'\in X.
\]
This is a semimetric analogue of $\mathbf{Met}$, a category of metric spaces.~\cite{deza2009encyclopedia}

\subsubsection{Path-representable distances as functors from graphs to semimetric spaces}
\label{subsec:pathrep_functor}

Given a weighted graph $G\in \mathbf{wGraph}$, a
\emph{path-representable distance} assigns to each pair of vertices
$v,w\in V(G)$ a non-negative number obtained as the cost of a
distinguished $v$--$w$ path, as in
\cref{sec:Path-representable distance}.
In general, many different path choice functions may represent the same graph,
and we do not attempt to make the full class of path-representable distances
functorial across different graphs.
In this work we focus on two distances on weighted graphs, introduced in \cref{def: distance_edge,def: distance_weight}.
For the precision, we reveal the graph where the distance is defined on, as $d_{\text{weight}}^G$.

\medskip
\noindent\textbf{A functorial distance construction.}
The weighted shortest-path distance is functorial with respect to the morphisms of
$\mathbf{wGraph}$:

\begin{lemma}
\label{lem:functorial_dweight}
Let $\varphi:G\to H$ be a morphism in $\mathbf{wGraph}$, i.e., a map
$\varphi:V(G)\to V(H)$ such that
\[
\{u,v\}\in E(G)\ \Rightarrow\ \{\varphi(u),\varphi(v)\}\in E(H),
\qquad
W_H(\varphi(u),\varphi(v))\le W_G(u,v).
\]
Define $d_{\text{weight}}^G$ and $d_{\text{weight}}^H$ by
\[
d_{\text{weight}}^G(x,y)=\min_{p:x\to y}\sum_{e\in E(p)}W_G(e),
\qquad
d_{\text{weight}}^H(x',y')=\min_{q:x'\to y'}\sum_{e\in E(q)}W_H(e).
\]
Then the same underlying vertex map
\[
\varphi_V:=\varphi:V(G)\longrightarrow V(H)
\]
is non-expansive:
\[
d_{\text{weight}}^H\bigl(\varphi_V(x),\varphi_V(y)\bigr)\ \le\ d_{\text{weight}}^G(x,y)
\qquad\text{for all }x,y\in V(G).
\]
Consequently,
\[
D_{d_{\text{weight}}}:\mathbf{wGraph}\to\mathbf{SemiMet},\qquad
G\mapsto (V(G),d_{\text{weight}}^G),\ \ \varphi\mapsto \varphi_V,
\]
is a well-defined functor.
\end{lemma}

\begin{proof}
Fix $x,y\in V(G)$ and let $p=(x=v_0,v_1,\dots,v_m=y)$ be any path in $G$.
Since $\varphi$ preserves edges, each $\{\varphi(v_{i-1}),\varphi(v_i)\}$ is an edge of $H$.
Hence $\varphi(p)=(\varphi(v_0),\varphi(v_1),\dots,\varphi(v_m))$ is a walk in $H$ from
$\varphi(x)$ to $\varphi(y)$.

By the weight-non-increasing property,
\[
W_H\bigl(\varphi(v_{i-1}),\varphi(v_i)\bigr)\ \le\ W_G(v_{i-1},v_i)
\quad\text{for each }i,
\]
so summing up gives
\[
\sum_{i=1}^m W_H\bigl(\varphi(v_{i-1}),\varphi(v_i)\bigr)
\ \le\
\sum_{i=1}^m W_G(v_{i-1},v_i).
\]
Because all edge weights are positive, any walk contains a path between the same endpoints
whose total weight is no larger.
Therefore,
\[
d_{\text{weight}}^H\bigl(\varphi(x),\varphi(y)\bigr)
\ \le\
\sum_{i=1}^m W_H\bigl(\varphi(v_{i-1}),\varphi(v_i)\bigr)
\ \le\
\sum_{i=1}^m W_G(v_{i-1},v_i).
\]
Finally, minimizing over all $G$-paths $p$ from $x$ to $y$ yields
$d_{\text{weight}}^H(\varphi(x),\varphi(y))\le d_{\text{weight}}^G(x,y)$.

Identity and composition holds because the assignment on morphisms is
the same underlying vertex map.
\end{proof}

\paragraph{On $d_{\text{edge}}$.} \label{rem:ded_nonfunctorial_thesis} In contrast to $d_{\text{weight}}$, the distance $d_{\text{edge}}$ is, in general,
\emph{not functorial} with respect to the morphisms of $\mathbf{wGraph}$.
A morphism $\varphi:G\to H$ may map $G$ into a target graph $H$ in which the images of
vertices admit additional edges not present in $G$, see \cref{fig:ded-nonfunctorial-example}.
Such new edges can change which paths are selected by the rule defining $d_{\text{edge}}$,
since $d_{\text{edge}}$ first minimizes the number of edges in a path and only then
minimizes the total weight among those minimizers.
Consequently, the non-expansiveness inequality
\[
  d_{\text{edge}}^{H}\bigl(\varphi(v),\varphi(w)\bigr)\le d_{\text{edge}}^{G}(v,w)
\]
need not hold for arbitrary morphisms in $\mathbf{wGraph}$.

\begin{figure}[t]
\centering
\begin{tikzpicture}[
  >=Latex,
  every node/.style={font=\small},
  vtx/.style={circle, draw, inner sep=1.2pt},
  wlab/.style={font=\scriptsize, fill=white, inner sep=1pt}
]

\coordinate (a) at (0,0);
\coordinate (b) at (1.3,1.8);
\coordinate (c) at (2.6,0);

\def\xshift{5.2}
\coordinate (x) at (\xshift+0,0);
\coordinate (y) at (\xshift+1.3,1.8);
\coordinate (z) at (\xshift+2.6,0);

\node[vtx] (A) at (a) {$a$};
\node[vtx] (B) at (b) {$b$};
\node[vtx] (C) at (c) {$c$};

\node[vtx] (X) at (x) {$x$};
\node[vtx] (Y) at (y) {$y$};
\node[vtx] (Z) at (z) {$z$};

\draw[thick] (A) -- node[wlab, left] {$1$} (B);
\draw[thick] (B) -- node[wlab, right] {$1$} (C);

\node[font=\small] at (1.3,2.35) {$G$};

\draw[thick] (X) -- node[wlab, left] {$1$} (Y);
\draw[thick] (Y) -- node[wlab, right] {$1$} (Z);
\draw[very thick] (X) -- node[wlab, below] {$10$} (Z);

\node[font=\small] at (\xshift+1.3,2.35) {$H$};

\draw[->, dashed] (A) to[bend left=10] (X);
\draw[->, dashed] (B) to[bend left=10] (Y);
\draw[->, dashed] (C) to[bend left=10] (Z);

\node[font=\small] at (\xshift/2+1.3,2.4) {$\varphi: G \to H$};

\node[align=center, font=\small] at (1.3,-0.95)
{$d_{\text{edge}}^{G}(a,c)=1+1=2$};

\node[align=center, font=\small] at (\xshift+1.3,-0.95)
{$d_{\text{edge}}^{H}(x,z)=10$};

\node[align=center, font=\small] at (\xshift/2+1.3,-1.65)
{$d_{\text{edge}}^{H}(\varphi(a),\varphi(c))=10\not\le 2=d_{\text{edge}}^{G}(a,c)$};

\end{tikzpicture}
\caption{A simple example showing that $d_{\text{edge}}$ is not functorial for morphisms in $\mathbf{wGraph}$: although $\varphi$ is edge-preserving and weight-non-increasing on edges of $G$, the additional edge $\{x,z\}$ in $H$ forces $d_{\text{edge}}^{H}(x,z)$ to use the single-edge route of weight $10$, so non-expansiveness fails.}
\label{fig:ded-nonfunctorial-example}
\end{figure}

This paper uses $d_{\text{edge}}$ primarily as a second distance on a \emph{fixed} graph $G$,
to be compared with $d_{\text{weight}}$ via barcode inclusion and total persistence difference.
No cross-graph functoriality of $d_{\text{edge}}$ is required for the main results in
the main text.
In particular, throughout this paper we compare the two distances
$d_{\text{weight}}^{G}\le d_{\text{edge}}^{G}$ pointwise on $V(G)\times V(G)$
for a fixed underlying graph $G$.
This pointwise inequality is a central ingredient in our later analysis of
inclusion between persistence barcodes and in the definition of total persistence difference.

\subsubsection{Vietoris--Rips filtration and persistent homology}

Having embedded graphs into semimetric spaces via the functors
$D_d$, we apply the standard Vietoris--Rips construction recalled in
the preliminaries in the main text.
Let $\mathbf{FiltSimp}$ denote the category of filtered finite
simplicial complexes indexed by $\mathbb{R}$, with morphisms given by
levelwise simplicial maps commuting with the inclusions.
To each semimetric space $(X,d)$ we associate the Vietoris--Rips
filtration
\[
  \mathrm{VR}(X,d)
  \;=\;
  \bigl(\mathrm{VR}(X,d)_\epsilon\bigr)_{\epsilon\in\mathbb{R}},
\]
where $\mathrm{VR}(X,d)_\epsilon$ has a simplex
$\sigma \subseteq X$ whenever $d(x,x')\le \epsilon$ for all
$x,x'\in\sigma$.
A non-expansive map $f \colon (X,d_X)\to (Y,d_Y)$ induces simplicial
maps $\mathrm{VR}(f)_\epsilon$ on each level, and these commute with
the inclusions of the filtrations.
Hence $\mathrm{VR}$ defines a functor
\[
  \mathrm{VR} \colon \mathbf{SemiMet} \longrightarrow \mathbf{FiltSimp}.
\]

For any fixed distance construction $d$ as above, composing
$\mathrm{VR}$ with $D_d$ yields a functor
\[
  \mathrm{VR}\circ D_d
  \colon
  \mathbf{wGraph}
  \longrightarrow
  \mathbf{FiltSimp},
  \qquad
  G \longmapsto \mathrm{VR}\bigl(V(G),d^G\bigr).
\]
In particular, we will apply this functorial pipeline to $d_{\text{weight}}$.
For $d_{\text{edge}}$, we apply Vietoris--Rips persistent homology after fixing a graph $G$
and equipping its vertex set with the semimetric $d_{\text{edge}}^{G}$.

For each $k\ge 0$, taking $k$-dimensional homology with
coefficients in $\mathbb{Z}_2$ gives a persistence module
\[
  H_k \circ \mathrm{VR}\circ D_d
  \colon
  \mathbf{wGraph}
  \longrightarrow
  \mathbf{PersMod},
\]
where $\mathbf{PersMod}$ denotes the category of pointwise finite
dimensional persistence modules indexed by $(\mathbb{R},\le)$.
Since our graphs are finite, the resulting Vietoris--Rips filtrations
are finite and the corresponding modules are tame, so they admit
persistence barcodes that enjoy the usual bottleneck stability properties, as guaranteed by the
general stability results in
\cite{bubenik2014categorification,mccleary2018bottleneck}.

From this perspective, one of the main pipelines in this paper is the composite functor
$G\mapsto H_1(\mathrm{VR}(V(G),d_{\text{weight}}^{G}))$.
In addition, for a fixed graph $G$ we also study
$H_1(\mathrm{VR}(V(G),d_{\text{edge}}^{G}))$ and compare it to the former.

The main text analyzes how pointwise inequalities between
such cost-dominated path-representable distances induce inclusions
between the corresponding Vietoris--Rips filtrations and barcodes,
while it also introduces the total persistence
difference as a scalar measure of how much the $1$-dimensional
persistence changes when the distance is altered on a fixed underlying
graph.
These constructions fit squarely within the classical functorial
framework of persistent homology on (semi)metric spaces, even though
the initial data consist of weighted graphs.

\end{document}